\documentclass[11pt]{article}

\usepackage{amssymb} \usepackage{amsthm}
\usepackage{amsmath}
\usepackage{latexsym} \usepackage{exscale}
\usepackage{epsfig}
\usepackage{epsf} \usepackage{graphicx}
\usepackage{enumerate}
\usepackage{color}
\usepackage[notref,notcite]{showkeys}
\newtheorem{theorem}{Theorem}[section]
\newtheorem{lemma}[theorem]{Lemma}

\newtheorem{remark}[theorem]{Remark}
\newtheorem{definition}[theorem]{Definition}

\newtheorem{asu}{Assumption}

\makeatletter \@addtoreset{equation}{section} \makeatother

\def\bt{\begin{theorem}}
\def\et{\end{theorem}}
\def\ba{\begin{array}}
\def\ea{\end{array}}
\def\bl{\begin{lemma}}
\def\el{\end{lemma}}

\def\bes{\begin{eqnarray}}
\def\ees{\end{eqnarray}}
\newcommand{\pp}{\partial_2}
\newcommand{\pn}{\partial_1}

\newcommand{\dspt}{\displaystyle}

\newcommand{\beq}{\begin{equation}}
\newcommand{\eeq}{\end{equation}}

\newcount\gnum
\def\G#1{
  \gnum=`#1
  \ifnum \gnum>`Z
     \advance\gnum by-`a
     \ifcase \gnum \alpha \or \beta \or \gamma \or \delta \or
        \epsilon\or\phi\or\theta\or\eta\or\iota\or\or\kappa\or
        \lambda\or\mu\or\nu\or\or\pi\or\xi\or\rho\or\sigma\or
        \tau\or\upsilon\or\or\omega\or\chi\or\psi\or\zeta\fi
  \else
     \advance\gnum by-`A
     \ifcase \gnum \Alpha \or \Beta \or \Gamma \or \Delta \or
        \Epsilon\or\Phi\or\Theta\or\Eta\or\Iota\or\or\Kappa\or
        \Lambda\or\Mu\or\Nu\or\or\Pi\or\Xi\or\Rho\or\Sigma\or
        \Tau\or\Upsilon\or\or\Omega\or\Chi\or\Psi\or\Zeta\fi
  \fi}

\def\({\left(\begin{array}{cccccc}}
\def\){\end{array}\right)}

\def\bes{\begin{eqnarray}}
\def\ees{\end{eqnarray}}

\begin{document}

\title{Global solution and singularity formation for the supersonic expanding wave of compressible Euler equations with radial symmetry}

\author{
Geng Chen\thanks{Department of Mathematics,
University of Kansas, Lawrence, KS
66045 ({\tt gengchen@ku.edu}).}
\and
Faris A. El-Katri\thanks{Department of Mathematics,
University of Kansas, Lawrence, KS
66045 ({\tt elkatri@ku.edu}).}
\and
Yanbo Hu\thanks{Department of Mathematics, Zhejiang University of Science and Technology, Hangzhou 310023, PR China ({\tt yanbo.hu@hotmail.com}).}
\and
Yannan Shen\thanks{Department of Mathematics,
University of Kansas, Lawrence, KS
66045 ({\tt yshen@ku.edu}).}
}
\date{}
\maketitle

\begin{abstract}
\noindent
In this paper, we define the rarefaction and compression characters for the supersonic expanding wave of the compressible Euler equations with radial symmetry. Under this new definition, we show that solutions with rarefaction initial data will not form shock in finite time, i.e. exist global-in-time as classical solutions. On the other hand, singularity forms in finite time when the initial data include strong compression somewhere. Several useful invariant domains will be also given.
\end{abstract}

2010\textit{\ Mathematical Subject Classification:} 76N15, 35L65, 35L67.

\section{Introduction}

We are interested in the global existence of smooth solution to the radially symmetric Euler equations with isentropic flow. The radially symmetric solution of Euler equations with $\gamma$-law pressure satisfies \cite{courant}
\begin{align}
	(r^m\rho)_t+(r^m\rho u)_r=&0\,,\label{Euler1}\\
	(r^m\rho u)_t+(r^m\rho u^2)_r+ r^m p_r=&0\,,\label{Euler2}\\
	p=&K\rho^\gamma.\label{Euler3}
\end{align}
Here $m=1, 2$ for flows with cylindrical or spherical symmetry, respectively, and $r>0$ denotes
the spatial coordinate. The symbols have their ordinary meaning: $\rho(r,t)$ is density, $u(r,t)$  is the particle velocity, $p(r,t)$ is the pressure and the adiabatic constant $\gamma>1$ for the isentropic gas.

The compressible Euler system is the most representative example of the nonlinear hyperbolic conservation laws, and its research has a very long history. A well-known fact is that, even for smooth and small initial data, its classical solutions may form gradient blowup in finite
time due to the highly nonlinear structures. The early study of breakdown of classical solutions for the compressible Euler equations were presented, among others, in \cite{Ali, Schen, John, Kong, lax2, Liu1, sideris}.

Later, the first author in \cite{G9,G10,CCZ,CY,CYZ, CPZ} gave a sequence of fairly complete results on the large data singularity formation theory for the isentropic and non-isentropic Euler equations in one space dimension, by adding an optimal time dependent lower bound on density and using the framework of Lax in \cite{lax2}.
Recently, the singularity formation result and  density lower bound have been extended to Euler equations with radial symmetry \eqref{Euler1}-\eqref{Euler3} in \cite{CCW}.

In \cite{Ch1}, Christodoulou introduced the geometric framework to
study the singularity formation of smooth solutions for the relativistic Euler equations in multi space dimensions.
The work by Christodoulou and Miao \cite{Ch2} revealed the fine geometric property of the singular hyperplane
for the multi-dimensional compressible Euler equations with isentropic irrotational and small perturbed initial data. In \cite{Luk1}, Luk and Speck studied the 2-d case with small but nonzero vorticity, even at the location of shock. Also see the related works in \cite{Di, Luk2} etc.

There are many recent results on the construction of shock formation which accurately describe the blowup process of the gradient of the solution.
We refer the reader to a series of related works  \cite{Vicol0, Vicol1, Vicol2, Vicol3, Vicol4, Vicol5}, in which the pre-shocks are accurately constructed.

On the other hand, it is meaningful and valuable to seek the conditions of the initial data to
guarantee the global existence of smooth (classical) solutions for the compressible Euler equations. A remarkable result is that classical solutions of isentropic Euler equations exist globally if and only if the initial data are rarefactive everywhere \cite{CCZ, CPZ,lax2}.
Here the rarefaction/compression character of solutions in two families are defined by the sign of the gradient on a pair of Riemann invariants. See other results on 1-d classical solutions in
\cite{CCZ, CY, LiBook, Li daqian, Zhu}.

For the multi-dimensional compressible Euler equations with isentropic gas, under the assumptions that the initial density is small smooth and the initial velocity is
smooth enough and forces particles to spread out, Grassin \cite{Grassin} established the global existence of smooth solutions. This global existence result was subsequently extended to the van der Waals gas in \cite{LM}. In a recent paper \cite{Lai}, Lai and Zhu established a global existence result for the two-dimensional axisymmetric Euler equations with a class of initial data that are constant state near the origin.
Also see earlier result by Godin \cite{Godin} on the lifespan of smooth solutions for the spherical symmetric Euler equations with initial data that are a small perturbation of a constant state.

Recently, Sideris \cite{sideris1} constructed a special of affine motions
to obtain the global existence of smooth solutions for the three-dimensional compressible isentropic Euler equations with a physical vacuum free boundary condition. In \cite{HJ}, Had${\rm \breve{z}}$i${\rm \acute{c}}$ and Jang proved that small perturbations of the expanding affine motions are globally-in-time stable, which leads to the global existence result for non-affine expanding solutions to the multi-dimensional isentropic Euler equations for the adiabatic exponent $\gamma\in(1,\frac{5}{3}]$. Their result was subsequently extended to $\gamma>1$ by Shkoller and Sideris in \cite{Shkoller-Sideris} and then to $\gamma>1$ for the non-isentropic system by Rickard, Had${\rm \breve{z}}$i${\rm \acute{c}}$ and Jang in \cite{Rickard2}.
Rickard \cite{Rickard1} discussed the global existence of the compressible isothermal Euler equations with heat transport by small perturbations on Dyson's isothermal affine solutions.
In \cite{PHJ}, Parmeshwar,  Had${\rm \breve{z}}$i${\rm \acute{c}}$ and Jang verified the global existence of expanding solutions to the isentropic Euler equations in the vacuum setting with small densities without relying on a perturbation.

In the current paper, we define the rarefaction and compression characters for expanding wave of the radially symmetric Euler equations with large initial data. We prove that solutions with rarefaction initial data will exist smoothly for any time. On the other hand, strong initial compression will form finite time singularities. This result is parallel to the equivalent condition on singularity formation for the 1-d isentropic Euler equations. For \eqref{Euler1}-\eqref{Euler3}, we expect there may be shock-free solutions including weak compression, see example for the non-isentropic Euler equations in \cite{CY,CCZ}.

The key ingredient of this paper is to find a pair of accurate gradient variables for \eqref{Euler1}-\eqref{Euler3}, whose signs define the rarefaction/compression characters of solutions in two different characteristic families, also named as R/C characters.
Then we use the Riccati type equations on these gradient variables to prove the desired results on global existence and singularity formation.

A natural idea, mimic the 1-d isentropic solution, is to use the gradient of Riemann variables, as in \cite{CCW}. However, such choice of gradient variables, fail to include the impact of the source term caused by radial symmetry.

The new idea in this paper comes from the belief that the stationary solution of  \eqref{Euler1}-\eqref{Euler3} is
neither rarefactive nor compressive. Then the gradient variables used to define R/C characters shall vanish in the stationary solution. So we may use the gradient along characteristic on some function which takes constant value in the stationary solution. More precisely, noticing that $r^m\rho u$ is constant in the stationary solution, we define the R/C characters  $\beta$ and $\alpha$ by
\beq\label{def_RC}
\beta=-\frac{\partial_2(r^m \rho u)}{r^m \rho c_1},\qquad
\alpha=-\frac{\partial_1(r^m \rho u)}{r^m \rho c_2},
\eeq
corresponding to the first and second characteristic families with speed
\[
	c_1 =u-\sqrt{p_\rho},\qquad c_2 =u+\sqrt{p_\rho},
\]
respectively,
where
\[
	\partial_2=\partial_t+c_2 \partial_r, \qquad \partial_1=\partial_t+c_1 \partial_r.
\]

Next, we introduce our main results verifying the effectiveness of this pair of variables for R/C characters.
We consider expanding supersonic waves satisfying the following initial condition.
\begin{asu}\label{asu_1}
Assume $1<\gamma<3$. For the initial data $(\rho_0(r),u_0(r))=(\rho(r,0), u(r,0))$ of \eqref{Euler1}-\eqref{Euler3}, there exists a uniform constant $C_0$ such that
\begin{equation}\label{inbo0}
0<\frac{2\sqrt{K\gamma}}{\gamma-1}\rho_0^{\frac{\gamma-1}{2}}(r)\leq u_0(r)\leq C_0,
\end{equation}
for any $r\in(b,\infty)$, with any given $b>0$.
\end{asu}
When we consider the initial-boundary value problem, as in Figure \ref{Fig:1}, with the 1-characteristic starting from $(b,0)$, denoted by $r=B_b(t)$, as the left boundary curve, we further assume that the boundary data satisfy the following property.
\begin{asu}\label{asu_2}
Assume $1<\gamma<3$. There exists a uniform constant $C_0$ such that
\begin{equation}\label{bbo0}
0<\frac{2\sqrt{K\gamma}}{\gamma-1}\rho_{b}^{\frac{\gamma-1}{2}}(t)\leq u_b(t)\leq C_0,
\end{equation}
for any $t\geq0$, where $(\rho_b(t), u_b(t))=(\rho(B_b(t), t), u(B_b(t), t))$.
\end{asu}

We know that inequalities \eqref{inbo0} form an invariant domain holding for any time before singularity formation, see \cite{CCW}. This similar property was first established by Chen in \cite{GQC} for studying the $L^\infty$ entropy solution.
So there will be a uniform upper bound on $\rho$ and $|u|$.

Now we define the rarefactive and compressive waves for smooth solutions. Here smooth solutions always mean $C^1$ solutions in this paper.
\begin{definition}
Consider smooth solutions \eqref{Euler1}-\eqref{Euler3} on some open domain on the $(r,t)$-plane, satisfying Assumption \ref{asu_1} (and also Assumption \ref{asu_2} if we consider the initial boundary problem).
\begin{itemize}
\item The solution is $1$-Rarefactive if $\beta\geq 0$ and $1$-Compressive if  $\beta\leq 0$.
\item
The solution is $2$-Rarefactive if $\alpha\geq0$ and $2$-Compressive if  $\alpha\leq 0$.
\end{itemize}
\end{definition}
For definitions of both rarefactive and compressive waves, we include $\alpha=0$ and $\beta=0$ for convenience.

The main results of this paper, for \eqref{Euler1}-\eqref{Euler3} satisfying Assumption \ref{asu_1} and \ref{asu_2}, can be summarized as:
\begin{itemize}
\item[1.] When the initial data are rarefactive, the solution is smooth globally: Theorem \ref{ex 1}, \ref{ex 2}, \ref{ex 3_0} and \ref{ex 3}
\item[2.] When the initial data include some strong compression, i.e. at some point $\alpha$ or $\beta$ is very negative, singularity form in finite time: Theorem \ref{thm_sing}.
\end{itemize}

Since physically $u_0(0)=0$, we note that the Assumption \ref{asu_1} is satisfied for the entire half line $\{r\geq 0\}$, only when $\rho_0(0)=0$.

For the existence result, we consider
both vacuum and non-vacuum cases at the origin. When $\rho_0(0)=0$, we assume that Assumption \ref{asu_1} is satisfied for the entire half line $\{r\geq 0\}$. For the non-vacuum case, we apply the
affine solution, inspired by \cite{CCW, sideris1}, to fill a small neighborhood centered at the origin.

It is well-known that the main difficulty in establishing the global existence of smooth solutions is
obtaining the a priori estimates of the solution and its gradient variables. First, an
appropriate density lower bound estimate is needed to extend the local solution to the entire domain. Fortunately, this result has already been proved in \cite{CCW}, following by the idea used in \cite{G9}. Also see \cite{CPZ,G10}

The key estimate on gradient variables comes from two invariant domains in the $(\alpha,\beta)$ plane, by studying the Riccati equations on $\alpha$ and $\beta$. On the other hand, the proof of singularity formation result heavily relies on the Riccati equations too. These equations will be used for the future study in other cases, not satisfying the Assumption \ref{asu_1}, including the imploding subsonic waves.

The rest of the paper is organized as follows. In Section \ref{S2}, we calculate the rarefactive and compressive characters. In Section \ref{S3}, we derive the Riccati equations of the R/C characters. In Section \ref{S4}, we establish the invariant domains for the smooth solution itself and the R/C characters. Moreover, we deduce
a density lower bound estimate which depends only on the time and is independent of the spatial variable $r$. In Section \ref{S5}, we prove the global existence of $C^1$-solution on the entire domain $t\geq0, r\geq0$ for rarefaction initial data. Finally, in Section \ref{S6}, we show that, when the initial data include strong compression somewhere, the smooth solution can form singularity in a finite time.

\section{The R/C characters}\label{S2}
When $\gamma>1$, we use the sound speed
\[h=\sqrt{p_\rho}=\sqrt{K\gamma}\,\rho^\frac{\gamma-1}{2}
\]
as the variable to take the place of $\rho$.
So, for smooth solutions, equations (\ref{Euler1})-(\ref{Euler2}) can be written as
\begin{align}
h_t+u\,h_r+\frac{\gamma-1}{2}\,h\, u_r=&-\frac{\gamma-1}{2}\,\frac{m}{r}\,u\,h\,,\label{h eq}\\
u_t+u\,u_r+\frac{2}{\gamma-1}\,h\,h_r=\,&0\,,\label{u eq}
\end{align}
with characteristic speeds
\bes
c_1=u-h\,,\quad c_2=u+h\,.\label{lrc}
\ees
Introduce the Riemann variables $w$ and $z$
\bes
w=u+\frac{2}{\gamma-1}h,\qquad
z=u-\frac{2}{\gamma-1}h\,.\label{def r}
\ees
Then we have the governing equations of $(w,z)$
\begin{align}
w_t+c_2\,w_r=&-\frac{m}{r}\,u\,h\,,\label{w eq}\\
z_t+c_1\,z_r=&\frac{m}{r}\,u\,h\,.\label{z eq}
\end{align}

%

Now we calculate the R/C characters $\alpha$ and $\beta$ defined in \eqref{def_RC}.
%
By \eqref{Euler2} and the definition of $\partial_1$, it is easy to calculate that, when $c_2\neq 0$ and $r>0$,
\begin{align}
\alpha=&-\frac{\partial_1(r^m \rho u)}{r^m \rho c_2}\nonumber\\
	=&-\frac{1}{r^m \rho (u+h)} \bigl( (r^m \rho u)_t + c_1  (r^m \rho u)_r \bigr)\nonumber\\
	=&-\frac{1}{r^m \rho (u+h)} \bigl( -(r^m \rho u^2)_r -  r^m p_r+ (u-h)  (r^m \rho u)_r \bigr)\nonumber\\
	=&-\frac{1}{r^m \rho (u+h)} \bigl( -r^m \rho u u_r -  r^m p_r-h  (r^m \rho u)_r \bigr)\nonumber\\
		=&-\frac{1}{r^m \rho (u+h)} \bigl( -r^m \rho u u_r -  r^m p_r-h  r^m \rho u_r -h  r^m \rho_r u-\frac{m}{r}h  r^m \rho u\bigr) \nonumber\\
	=& u_r +\frac{2}{\gamma-1}h_r+\frac{m}{r} \frac{hu}{c_2}\nonumber,
\end{align}
where in the last step, we used
\begin{align}
-r^mp_r-h  r^m \rho_r u
=&-r^m(K\gamma \rho^{\gamma-1}\rho_r+h\rho_r u)\nonumber\\
=&-r^m(h^2+h u)\rho_r\nonumber\\
=&-\frac{2}{\gamma-1}r^m\rho(u+h)h_r.\nonumber
\end{align}

Similarly,  when $c_1\neq 0$ and $r>0$,, we have
\begin{align}
\beta=&-\frac{\partial_2(r^m \rho u)}{r^m \rho c_1}
	=-\frac{1}{r^m \rho  (u-h)} \bigl( (r^m \rho u)_t + c_2  (r^m \rho u)_r \bigr)\nonumber\\
	=&-\frac{1}{r^m \rho  (u-h)} \bigl( -(r^m \rho u^2)_r -  r^m p_r+ (u+h)  (r^m \rho u)_r \bigr)\nonumber\\
	=& u_r -\frac{2}{\gamma-1}h_r-\frac{m}{r} \frac{hu}{c_1}. \nonumber
\end{align}
In summary, when $c_1\neq0$, $c_2\neq 0$ and $r>0$, for smooth solution,  \eqref{def_RC} is equivalent to
\bes
\alpha=u_r+\frac{2}{\gamma-1}h_r+\frac{m}{r}\frac{\,h\, u}{{c_2}},\label{def alpha}\\
\beta=u_r-\frac{2}{\gamma-1} h_r-\frac{m}{r}\frac{ \,h\,u}{{c_1}}.\label{def beta}
\ees
We remark that $\alpha$ and $\beta$  are different from derivatives of Riemann variables, i.e. $w_r$, $z_r$.

\section{Riccati equations}\label{S3}

Now, we can derive the following Riccati equations for Euler equations.
\begin{lemma}\label{lemma_ric}
For smooth solution of \eqref{Euler1}-\eqref{Euler3}, we have the following Riccati type equations on $\alpha$ and $\beta$ defined in \eqref{def alpha} and \eqref{def beta}, when $r>0$ and $c_1c_2\neq 0$,
\bes\label{beta_eq}
\pn\beta=-\frac{1+\gamma}{4}\beta^2-\frac{3-\gamma}{4}\alpha\beta+A_1\alpha-B_1\beta,
\ees
where
\beq\label{A1Def}
A_1=
\frac{m c_2}{2rc_1^2}(\frac{\gamma-1}{2}u^2-h^2),\eeq
\beq\label{B1Def}
B_1=\frac{m}{rc_1^2}\bigg(\frac{\gamma-1}{4}u^3-\frac{1}{2}h^3
-\frac{\gamma-1}{4}u^2h+
\frac{1}{2}uh^2+
\frac{hu{c_1}}{{c_2}}(h+\frac{\gamma-1}{2}u)\bigg),\nonumber
\eeq
and \bes\label{alphacd1}
\pp\alpha=-\frac{\gamma+1}{4}\alpha^2-\frac{3-\gamma}{4}\alpha\beta+A_2\beta-B_2\alpha,
\ees
where
\beq\label{alphacd2}
A_2=
\frac{m c_1}{2rc_2^2}(\frac{\gamma-1}{2}u^2-h^2),
\eeq
\beq\label{alphacd3}
B_2=\frac{m}{rc_2^2}\bigg(\frac{\gamma-1}{4}u^3+\frac{1}{2}h^3+
\frac{\gamma-1}{4}u^2h+
\frac{1}{2}uh^2+\frac{hu{c_2}}{{c_1}}(h-\frac{\gamma-1}{2}u)\bigg).
\eeq
\end{lemma}
Equations \eqref{beta_eq} and \eqref{alphacd1} are homogeneous, i.e. the right hand sides of \eqref{beta_eq} and \eqref{alphacd1} vanish when $\alpha$ and $\beta$ both vanish. This is because we find $\alpha$ and $\beta$ from the stationary solution, so when $\alpha=\beta=0$, their time derivatives also vanish. This property is crucial for us to prove several important invariant domains on gradient variables, which serve as our basis to study the existence of classical solutions and singularity formation.
\begin{proof}
We first calculate the equation of $\alpha$ by \eqref{def alpha}
\begin{align}
\pp\alpha=&(u_r+\frac{2}{\gamma-1}h_r+\frac{m}{r}\frac{\,h\, u}{{c_2}})_t+{c_2}(u_r+\frac{2}{\gamma-1}h_r+\frac{m\,h\, u}{rc_2})_r\nonumber\\
=&u_{rt}+\frac{2}{\gamma-1} h_{rt}+(u+h)u_{rr}+(u+ h)\frac{2}{\gamma-1} h_{rr}\nonumber\\
&+\frac{m}{r}(\frac{hu}{{c_2}})_t+{c_2}(\frac{m}{r}\frac{hu}{{c_2}})_r.\label{Sprime}
\end{align}
It follows by (\ref{h eq})-(\ref{u eq}) that
\bes
h_{rt}+uh_{rr}+u_rh_r+\frac{\gamma-1}{2}u_rh_r+\frac{\gamma-1}{2}u_{rr}h =-\frac{\gamma-1}{2}(\frac{m}{r}uh)_r,\label{z_x}
\ees
and
\bes
u_{rt}+uu_{rr}+u_r^2+\frac{2}{\gamma-1}h_r^2+\frac{2}{\gamma-1}h_{rr}h=0.\label{u_x}
\ees
By (\ref{z_x})-(\ref{u_x}), one obtains
\begin{align}
&u_{rt}+\frac{2}{\gamma-1}h_{rt}+uu_{rr}+hu_{rr}+u\frac{2}{\gamma-1}h_{rr}+\frac{2}{\gamma-1}hh_{rr} \nonumber\\
=& -\frac{2}{\gamma-1}h_r^2-u_r^2-\frac{\gamma+1}{\gamma-1}u_r h_r-(\frac{m}{r}uh)_r. \label{u_rt}
\end{align}
Substituting (\ref{u_rt}) into (\ref{Sprime}) yields
\begin{align}
\pp\alpha=&\bigg(\frac{m}{r}(\frac{uh}{{c_2}})_t-\frac{muh}{rc_2}{c_2}_r\bigg) +\bigg(-\frac{\gamma+1}{\gamma-1}u_rh_r-u_r^2-\frac{2}{\gamma-1}h_r^2\bigg) \nonumber \\
\triangleq &(i)+(ii). \nonumber
\end{align}

Furthermore, we apply (\ref{def alpha}) and (\ref{def beta}) to get
\begin{align*}
\alpha+\beta =& 2u_r+\frac{m}{r}uh(\frac{1}{{c_2}}-\frac{1}{{c_1}}),\\
\alpha-\beta =& \frac{4}{\gamma-1}h_r+\frac{m}{r}uh(\frac{1}{{c_1}}+\frac{1}{{c_2}}),
\end{align*}
which mean that
\begin{align}
h_r =& \frac{\gamma-1}{4}\bigg(\alpha-\beta-\frac{m}{r}uh(\frac{1}{{c_1}}+\frac{1}{{c_2}})\bigg),\label{hr}\\
u_r =& \frac{1}{2}\bigg(\alpha+\beta-\frac{m}{r}uh(\frac{1}{{c_2}}-\frac{1}{{c_1}})\bigg).\label{ur}
\end{align}

By (\ref{hr}) and (\ref{ur}), we find that
\begin{align}
&(ii)
=-(u_r+h_r)(u_r+\frac{2}{\gamma-1}h_r)\nonumber\\
=&-\bigg(\frac{3-\gamma}{4}\beta+\frac{\gamma+1}{4}\alpha+ \frac{m}{r}\frac{uh}{{c_2}{c_1}}(h-\frac{\gamma-1}{2}u)\bigg)\big(\alpha-\frac{m}{r}\frac{uh}{{c_2}}\big) \nonumber\\
=&-\frac{\gamma+1}{4}\alpha^2-\frac{3-\gamma}{4}\alpha\beta +(\frac{3-\gamma}{4}\beta+\frac{\gamma+1}{4}\alpha)\frac{m}{r}\frac{uh}{c_2}\nonumber\\
&-\frac{m}{r}\frac{uh(h-\frac{\gamma-1}{2}u)}{{c_1}{c_2}}\alpha +\frac{m^2}{r^2}\frac{u^2h^2}{c_1c^2_2}(h-\frac{\gamma-1}{2}u), \nonumber
\end{align}
and
\begin{align}
(i)=&\frac{m}{r}(\frac{hu}{{c_2}})_t-\frac{m}{r}\frac{uh}{{c_2}}{c_2}_r\nonumber\\
=&\frac{m}{r}\frac{1}{{c_2}}(uh)_t-\frac{1}{c_2^2}{c_2}_t \frac{m}{r}uh-\frac{m}{r}\frac{uh}{{c_2}}{c_2}_r\nonumber\\
=&\frac{m}{r c_2^2}\big({c_2}(uh)_t-uh{c_2}_t-uh{c_2}{c_2}_r\big),\nonumber
\end{align}
where
\begin{align}
&{c_2}(uh)_t-uh{c_2}_t-uh{c_2}{c_2}_r\nonumber\\
 =&u^2h_t
 +h^2u_t-hu^2u_r-u^2hh_r-h^2uu_r-uh^2h_r\nonumber\\
=&-h_r\big(u^3+\frac{2}{\gamma-1}h^3+uh^2+u^2h\big)-u_r\big(2h^2u+\frac{\gamma+1}{2}u^2h\big) -\frac{\gamma-1}{2}u^3ha\nonumber\\
=&\frac{\gamma-1}{4}\big(\alpha-\beta-\frac{m}{r}uh(\frac{1}{{c_1}}+\frac{1}{{c_2}})\big)
\big(-u^3-\frac{2}{\gamma-1}h^3-uh^2-u^2h\big)\nonumber\\
&+\frac{1}{2}\big(\alpha+\beta-\frac{m}{r}uh(\frac{1}{{c_2}}-\frac{1}{{c_1}})\big)
\big(-2h^2u-\frac{\gamma+1}{2}u^2h\big)-\frac{\gamma-1}{2}\frac{m}{r}u^3h.\nonumber
\end{align}
Here we used (\ref{hr})-(\ref{ur}) in the last step. Then regarding $(i)+(ii)$ as a polynomial of $\alpha,\beta$, one can easily derive (\ref{alphacd1})-(\ref{alphacd3}).

For the equation of $\beta$, we first calculate
\begin{align}
\pn\beta=&(u_r-\frac{2}{\gamma-1} h_r-\frac{m}{r}\frac{ h\,u}{{c_1}})_t+{c_1}(u_r-\frac{2}{\gamma-1} h_r-\frac{m}{r}\frac{ h\,u}{{c_1}})_r\nonumber\\
=&u_{rt}-\frac{2}{\gamma-1} h_{rt}+(u-h)u_{rr}-(u- h)\frac{2}{\gamma-1} h_{rr}\nonumber\\
&-\frac{m}{r}(\frac{hu}{{c_1}})_t-{c_1}(\frac{m}{r}\frac{hu}{{c_1}})_r.\label{Sprime2}
\end{align}
Putting (\ref{u_rt}) into (\ref{Sprime2}) arrives at
\begin{align}
\pn\beta=&\bigg(-\frac{m}{r}(\frac{hu}{{c_1}})_t+\frac{m}{r}\frac{hu}{{c_1}}{c_1}_r\bigg) +\bigg(\frac{\gamma+1}{\gamma-1}u_rh_r-u_r^2-\frac{2}{\gamma-1}h_r^2\bigg)\nonumber\\
\triangleq&(iii)+(iv). \nonumber
\end{align}
We utilize (\ref{hr}) and (\ref{ur}) again to gain
\begin{align}
&(iii)
=-(u_r-h_r)(u_r-\frac{2}{\gamma-1}h_r)\nonumber\\
=&-\big(\frac{3-\gamma}{4}\alpha+\frac{\gamma+1}{4}\beta+ \frac{m}{r}\frac{uh}{{c_2}{c_1}}(h+\frac{\gamma-1}{2}u)\big) \big(\beta+\frac{m}{r}\frac{uh}{{c_1}}\big)\nonumber\\
=&-\frac{\gamma+1}{4}\beta^2-\frac{3-\gamma}{4}\alpha\beta -(\frac{3-\gamma}{4}\alpha+\frac{\gamma+1}{4}\beta)\frac{m}{r}\frac{uh}{c_1}\nonumber\\
&-\frac{m}{r}\frac{uh(h+\frac{\gamma-1}{2}u)}{{c_1}{c_2}}\beta -\frac{m^2u^2h^2}{r^2c_1^2c_2}(h+\frac{\gamma-1}{2}u),
\nonumber
\end{align}
and
\begin{align}
(iv)
=&-\frac{m}{r}(\frac{hu}{{c_1}})_t+\frac{m}{r}\frac{uh}{{c_1}}{c_1}_r\nonumber\\
=&-\frac{1}{{c_1}}a(uh)_t+\frac{1}{c_1^2}{c_1}_t \frac{m}{r}uh+\frac{m}{r}\frac{uh}{{c_1}}{c_1}_r\nonumber\\
=&-\frac{m}{rc_1^2}\big({c_1}(uh)_t-uh{c_1}_t-uh{c_1}{c_1}_r\big),\nonumber
\end{align}
where
\begin{align}
&{c_1}(uh)_t-uh{c_1}_t-uh{c_1}{c_1}_r\nonumber\\
=&u^2h_t
 -h^2u_t-hu^2u_r+hu^2h_r+h^2uu_r-h^2uh_r\nonumber\\
=&h_r\big(\frac{2}{\gamma-1}h^3-u^3-h^2u+hu^2\big)+u_r\big(2h^2u-\frac{\gamma+1}{2}hu^2\big) -\frac{\gamma-1}{2}u^3ha\nonumber\\
=&\frac{\gamma-1}{4}\big(\alpha-\beta-\frac{m}{r}uh(\frac{1}{{c_1}}+\frac{1}{{c_2}})\big)
\big(-u^3+\frac{2}{\gamma-1}h^3-h^2u+hu^2\big)\nonumber\\
&+\frac{1}{2}\big(\alpha+\beta-\frac{m}{r}uh(\frac{1}{{c_2}}-\frac{1}{{c_1}})\big)
\big(2h^2u-\frac{\gamma+1}{2}hu^2\big)-\frac{\gamma-1}{2}\frac{m}{r}u^3h.\nonumber
\end{align}
By regarding $(iii)+(iv)$ as a polynomial of $\alpha,\beta$, we get \eqref{beta_eq}.
\end{proof}

For future use, we calculate $\partial_1h $ and $\partial_2h$:
\begin{align}
\partial_2h = \partial_th+(u+h)\partial_rh\nonumber\\
=h_t+uh_r+hh_r.
\end{align}
Therefore, by \eqref{h eq}, we know that
\begin{align}
\partial_2h=-\frac{\gamma-1}{2}\frac{m}{r}uh-\frac{\gamma-1}{2}hu_r+hh_r,
\end{align}
which yields
\begin{align}\label{p2h}
\partial_2h=-\frac{\gamma-1}{2}\frac{m}{rc_1}u^2h-\frac{\gamma-1}{2}h\beta.
\end{align}
Via a similar computation, we arrive at
\begin{align}\label{p1h}
\partial_1h = -\frac{\gamma-1}{2}\frac{m}{rc_2}u^2h-\frac{\gamma-1}{2}h\alpha.
\end{align}

\section{Invariant domains}\label{S4}
Before studying the solution in the whole domain $(r,t)\in\mathbb R^+\times \mathbb R^+$, we first consider the following two problems on the domain $\Omega$ in the $(r,t)$-plane.

\begin{definition} We define two problems for \eqref{Euler1}-\eqref{Euler3}.
\begin{itemize}
\item Problem 1: Cauchy problem on domain of dependence $\Omega$ on the $(r,t)$-plane with base $t=0$ and $(b,\infty)$ for any $b>0$, i.e. domain to the right of the 2-characteristic starting at  $(b,0)$.
\item Problem 2: Boundary value problem on the half Goursat problem on the domain $\Omega$ on the $(r,t)$-plane to the right of the 1-characteristic boundary $r=B_b(t)$ starting at the point $(b,0)$.
\end{itemize}
\end{definition}

Later, using results on these two problems, we will prove the global existence of some classical solution on the entire half line $r\in[0,\infty)$. The construction is divided into two cases: $\rho(0,t)=0$ and $\rho(0,t)>0$.
Specifically, to construct a smooth solution with positive density at the origin, we apply the affine solution to fill the domain between the line $r=0$ and the 1-characteristic curve $r=B_b(t)$.

\begin{figure}[h!]
\centering
\scalebox{0.45}{\includegraphics{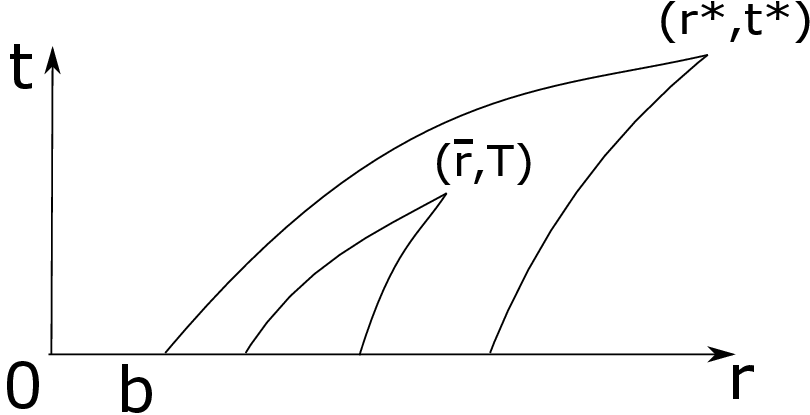}}
\scalebox{0.45}{\includegraphics{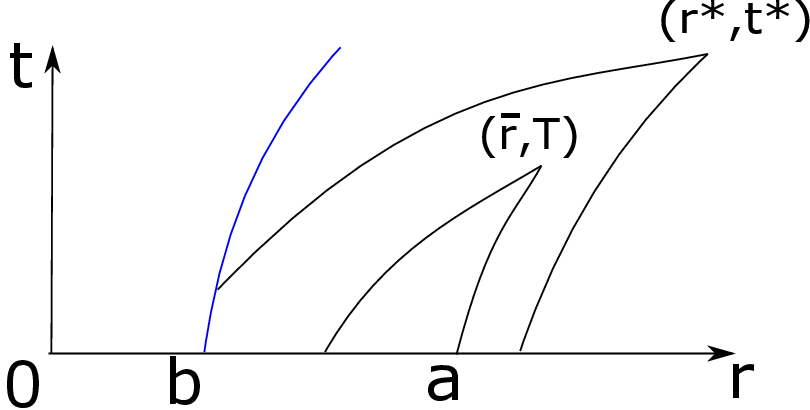}}
 \caption{Proof of Lemma \ref{lem1} and \ref{lem2}. \label{Fig:1} }
\normalsize
\end{figure}

We first review an invariant domain on $(u,\rho)$ or equivalently on $(u,h)$.

\begin{theorem}\label{thm_bounds}
For $1<\gamma\leq3$, when the initial data satisfy the initial Assumption \ref{asu_1}, any smooth solution of Problem 1 in the domain of dependence $\Omega$ based on $(b,\infty)$, with $b>0$,
satisfies
\begin{equation}\label{inbo1}
\frac{2\sqrt{K\gamma}}{\gamma-1}\rho^{\frac{\gamma-1}{2}}(r,t)\leq u(r,t)\leq 2C_0.
\end{equation}
If the initial Assumption \ref{asu_1} and the boundary Assumption \ref{asu_2} are satisfied, then inequality \eqref{inbo1} holds for any smooth solution of Problem 2 in the domain of dependence $\Omega$.
Moreover, when $b=0$, assume initially \eqref{inbo0} holds for any $r\in(0,\infty)$ and $
\rho_0(0)=u_0(0)=0,$
then the inequality \eqref{inbo1} holds on the domain $\{r>0,t\geq0\}$, with $\rho(0,t)=u(0,t)=0$.
\end{theorem}
\begin{proof}
For Problem 1, one can directly use the proof of Theorem 2.1 in \cite{CCW} to obtain that the inequality \eqref{inbo1}  holds on the domain $\Omega$. We here give the proof of \eqref{inbo1} to Problem 2
for the sake of completeness.

Thanks to the result of Problem 1, it suffices to verify \eqref{inbo1} on the region $\Omega_0$ bounded by $B_b(t)$ and the 2-characteristic starting at point $(b,0)$. Recalling the definitions of $w$ and $z$, we only need to prove the following inequalities
\begin{equation}\label{inbo r}
z\geq0,\quad w\leq 2C_0,
\end{equation}
on $\Omega_0$. We shall show $z<w$ (i.e. $\rho>0$) for $1<\gamma<3$ at the end of this section.

Assume that $z(\bar{r},\bar{t})<0$ at some point $(\bar{r},\bar{t})\in\Omega_0$ with $\bar{t}>0$. Denote the 1-characteristic $r=r_1(t)$ through the point $(\bar{r},\bar{t})$ by $l_1$. Due to $z(r_1(0),0)\geq0$ by \eqref{inbo0}, there exist some times $0\leq\tilde{t}<\hat{t}\leq \bar{t}$ such that on $l_1$, there hold $z(r_1(t),t)\geq0$ for $t\in[0,\tilde{t})$, $z(r_1(\tilde{t}),\tilde{t})=0$, and $z(r_1(t),t)<0$ for $t\in(\tilde{t},\hat{t}]$. Thus one has $\partial_1z(r_1(\tilde{t}),\tilde{t})<0$. On the other hand, we recall \eqref{z eq} to find that
$$
\partial_1z(r_1(\tilde{t}),\tilde{t})=\frac{m}{r}u(r_1(\tilde{t}), \tilde{t})h(r_1(\tilde{t}),\tilde{t})\geq0,
$$
a contradiction. Hence $z(r,t)\geq0$ in $\Omega_0$.

According to $z(r,t)\geq0$, we have $u(r,t)\geq \frac{2}{\gamma-1}h(r,t)\geq0$ in $\Omega_0$.
For any point $(r,t)\in\Omega_0$, we draw the 2-characteristic $r=r_2(t)$ through $(r,t)$ up to the boundary $B_b(t)$ at $Q$. From \eqref{w eq}, one gets
$$
\partial_2w(r_2(t),t)=-\frac{m}{r}u(r_2(t), t)h(r_2(t),t)\leq0,
$$
which implies that $w(r_2(t),t)\leq w(Q)$. Thus we use the boundary condition \eqref{bbo0} to obtain $w(r_2(t),t)\leq 2C_0$.
\end{proof}


Once the inequality \eqref{inbo1} and $z<w$ hold, it is easy to check that, when $r>b\geq0$,
\begin{equation}\label{inbo2}
u-\frac{2}{\gamma-1}h\geq 0,\qquad c_1>0,\qquad c_2>0,
\end{equation}
so
\beq\label{A12}
A_1, A_2\geq 0,
\eeq
where $A_1$ and $A_2$ are defined in \eqref{A1Def} and \eqref{alphacd2}.

Using \eqref{A12}, we next prove that
the set $\{\min\{\alpha,\beta\}\geq 0\}$ is an invariant domain in Problem 1 and 2 in the following lemma.

\begin{lemma}\label{lem1}
Assume $1<\gamma<3$.
Consider smooth solution on $t\in[0,T_0]$ for Problem 1 or 2 on
$\Omega$, satisfying the Assumption \ref{asu_1} on $(b,\infty)$, with $b>0$. For Problem 2, we also assume Assumption \ref{asu_2} on the left boundary 1-characteristic $B_b(t)$.
One can show that  if
$$
\min_{[b,\infty)}(\alpha,\beta)(r,0)\geq0,
$$
and
$\min_{t\geq0}\alpha(B_b(t),t)\geq0$ for Problem 2,
then for smooth solution
$$
\min_{\Omega\cap\{t\leq T_0\}}(\alpha,\beta)(r,t)\geq0,
$$
i.e. $\{(\alpha,\beta)|\min(\alpha,\beta)\geq0\}$ is an invariant domain on time.
\end{lemma}
\begin{remark}
The idea of proof is base on the observation that $\{\min(\alpha,\beta)\geq0\}$ is an invariant domain on the $(\alpha, \beta)$-plane. In fact, when $\beta=0$ and $\alpha\geq 0$, $\partial_1\beta=A_1\alpha\geq 0$, and
when $\alpha=0$ and $\beta\geq 0$, $\partial_2\alpha=A_2\beta\geq 0$. So it is very easy to prove that, $\{\min(\alpha,\beta)>0\}$ is an invariant domain using the similar method as in \cite{CCW} or \cite{G9}. To obtain a more general version of invariant domain on $\{\min(\alpha,\beta)\geq 0\}$, one needs to introduce a small perturbation $\varepsilon e^{\widehat{M}t}$.
\end{remark}
\begin{proof}
According to Theorem \ref{thm_bounds} and the precise expressions of $A_i, B_i$, we know that there exists a positive constant $\widehat{K}$ such that
\begin{align}\label{dd15}
0\leq A_{1,2}(r,t)\leq \widehat{K},\quad |B_{1,2}(r,t)|\leq \widehat{K},\quad \forall\ (r,t)\in\Omega,
\end{align}
where $\widehat K$ depends on $b$. Here we note that
$$
\frac{h}{c_1}=\frac{h}{u-h}= \frac{h}{u-\frac{2}{\gamma-1}h +\frac{3-\gamma}{\gamma-1}h} \leq \frac{h}{\frac{3-\gamma}{\gamma-1}h}=\frac{\gamma-1}{3-\gamma},
$$
and
\begin{align}\label{dd1}
\frac{u}{c_1}=\frac{u}{u-h}=1+\frac{h}{c_1}\leq \frac{2}{3-\gamma},
\end{align}
which have positive upper bounds when $1<\gamma<3$.

Set $\widehat{M}=2\widehat{K}+2$ and $\varepsilon>0$ is an arbitrary small number such that $\varepsilon e^{\widehat{M}T_0}<1$. We now introduce two new variables for $t\leq T_0$
$$
X=\alpha+\varepsilon e^{\widehat{M}t},\quad Y=\beta+\varepsilon e^{\widehat{M}t}.
$$
In view of Lemma \ref{lemma_ric}, one can derive the governing system of $(X,Y)$
\begin{align}\label{X eq}
\pp X=&\bigg\{-\frac{\gamma+1}{4}(X-2\varepsilon e^{\widehat{M}t})-\frac{3-\gamma}{4}(Y-\varepsilon e^{\widehat{M}t})-B_2\bigg\}X \nonumber \\
&+\bigg\{\frac{3-\gamma}{4}\varepsilon e^{\widehat{M}t} +A_2\bigg\}Y +\varepsilon e^{\widehat{M}t}\bigg\{\widehat{M}-A_2-B_2-\varepsilon e^{\widehat{M}t}\bigg\},
\end{align}
and
\begin{align}\label{Y eq}
\pn Y=&\bigg\{-\frac{\gamma+1}{4}(Y-2\varepsilon e^{\widehat{M}t})-\frac{3-\gamma}{4}(X-\varepsilon e^{\widehat{M}t})-B_1\bigg\}Y \nonumber \\
&+\bigg\{\frac{3-\gamma}{4}\varepsilon e^{\widehat{M}t} +A_1\bigg\}X  +\varepsilon e^{\widehat{M}t}\bigg\{\widehat{M}-A_1-B_1-\varepsilon e^{\widehat{M}t}\bigg\}.
\end{align}
By the choice of $\widehat{M}$ and $\varepsilon$, it is observed that
\begin{align}\label{aa1}
\widehat{M}-A_i-B_i-\varepsilon e^{\widehat{M}t}\geq \widehat{M}-2\widehat{K}-\varepsilon e^{\widehat{M}T_0}>1.
\end{align}

We now apply the contradiction argument to show $\{(X,Y)|\min(X,Y)>0\}$ is an invariant domain for $t\leq T_0$. We first see by the initial or initial boundary value conditions that $X(r,0)>0, Y(r,0)>0$ ($r\in[b,\infty)$) and $X(B_b(t),t)>0, Y(B_b(t),t)>0$ ($t\geq0$) for Problem 2.
Suppose that the region $\{(X,Y)|\min(X,Y)>0\}$ is not an invariant domain; that is, there exists some time, such that $X(r^*,t^*)=0$ or $Y(r^*,t^*)=0$, at some point $(r^*,t^*)$ with $0<t^*\leq T_0$. Because the wave speed is bounded on $[0,t^*]$, then we can find the characteristic triangle with vertex $(r^*,t^*)$ and lower boundary on the initial line $t=0$, denoted by $\Pi_0\in \Omega$. If $(r^*,t^*)\in\Omega_0$ for Problem 2, then $\Pi_0$ is the characteristic quadrangle bounded by the 1- and 2-characteristics through $(r^*,t^*)$, lower boundary on the initial line $t=0$, and part of curve $B_b(t)$. See Figure \ref{Fig:1}.

In turn, we can find the first time $T\leq T_0$ such that $X(T)=0$ or $Y(T)=0$ in $\Pi_0$.

\paragraph{Case 1:} At time $T$, $Y = 0$, and $X \geq 0$.

In this case, we see by \eqref{Y eq} and \eqref{aa1} that there holds in the interval $[0, T)$
\begin{align}
\partial_1Y
> &\bigg\{-\frac{\gamma+1}{4}(Y-2\varepsilon e^{\widehat{M}t})-\frac{3-\gamma}{4}(X-\varepsilon e^{\widehat{M}t})-B_1\bigg\}Y,
\end{align}
which indicates that
$$
\frac{\partial_1Y}{Y} > -\frac{\gamma+1}{4}(Y-2\varepsilon e^{\widehat{M}t})-\frac{3-\gamma}{4}(X-\varepsilon e^{\widehat{M}t})-B_1.
$$
Integrating the above along the 1-characteristic $r=r_1(t)$ from $0$ to $s< T$ yields
\begin{align*}
&Y(s)=Y(r_1(s),s)> Y(r_1(0),0) \\
&\times\exp\bigg\{\int^s_0\bigg(-\frac{\gamma+1}{4}(Y-2\varepsilon e^{\widehat{M}\tau})-\frac{3-\gamma}{4}(X-\varepsilon e^{\widehat{M}\tau})-B_1\bigg) (r_1(\tau),\tau)d\tau \bigg\}.
\end{align*}
Thus implying that $s$ cannot be finite. Contradiction.

\paragraph{Case 2:} At time $T$, $X = 0$, and $Y \geq 0$.

In this case, we consider the interval $[0, T)$ and apply \eqref{X eq} and \eqref{aa1} again to obtain
\begin{align*}
\partial_2X&>\bigg\{-\frac{\gamma+1}{4}(X-2\varepsilon e^{\widehat{M}t})-\frac{3-\gamma}{4}(Y-\varepsilon e^{\widehat{M}t})-B_2\bigg\}X,
\end{align*}
and subsequently
\[  \frac{\partial_2X}{X} > -\frac{\gamma+1}{4}(X-2\varepsilon e^{\widehat{M}t})-\frac{3-\gamma}{4}(Y-\varepsilon e^{\widehat{M}t})-B_2.\]
One integrates the above along the 2-characteristic $r=r_2(t)$ from $s_0$ to $s< T$ to acquire
\begin{align*}
&X(s)=X(r_2(s), s)
> X(r_2(s_0), s_0) \\
&\times\exp\bigg\{\int^s_{s_0}\bigg(-\frac{\gamma+1}{4}(X-2\varepsilon e^{\widehat{M}\tau})-\frac{3-\gamma}{4}(Y-\varepsilon e^{\widehat{M}\tau})-B_2\bigg) (r_2(\tau),\tau)d\tau\bigg\},
\end{align*}
where $s_0=0$ if $(r_2(s), s)\in\Omega\setminus\Omega_0$, while $s_0$ is determined by $r_2(s_0)=B_b(s_0)$ if $(r_2(s), s)\in\Omega_0$. In view of the initial and boundary value conditions, we observe that
$s$ cannot be finite, which leads a contradiction.

Hence we have
$$
X=\alpha(r,t)+\varepsilon e^{\widehat{M}t}>0,\quad Y=\beta(r,t)+\varepsilon e^{\widehat{M}t}>0,
$$
for $t\leq T_0$ and any $\varepsilon$ satisfying $\varepsilon e^{\widehat{M}T_0}<1$. By the arbitrariness of $\varepsilon$, one obtains
$$
\alpha(r,t)\geq0,\quad \beta(r,t)\geq0.
$$
The proof of the lemma is complete.
\end{proof}

Next, we prove another invariant domain on the upper bounds of $\alpha$ and $\beta$, which show that the maximum rarefaction is bounded if it is initially bounded. A similar invariant domain was first established for the 1-d problem in \cite{G9,G10}, then extended to the radially symmetric solution in \cite{CCW}. Using the new coordinates, we get a much better estimate in the following lemma than the one in \cite{CCW}.

\begin{lemma}\label{lem2}
Assume $1<\gamma<3$. Consider smooth solution on $t\in[0,T_0]$ for Problem 1 or 2 on
$\Omega$, satisfying the Assumption \ref{asu_1} on $(b,\infty)$, with $b>0$.  For Problem 2, we also assume Assumption \ref{asu_2} on the left boundary 1-characteristic $B_b(t)$.
For any smooth solution in $t\in[0,T_0]$ satisfying the initial condition
\beq\label{lem2_2}
\min_{r\in[b,\infty)}(\alpha,\beta)(r,0)\geq0,\quad\hbox{and}\quad
\min_{t\geq0}\alpha(B_b(t),t)\geq0\quad \hbox{for Problem 2,}
\eeq
then for any $M>0$, if
$$
\max_{r\in[b,\infty)}(\alpha,\beta)(r,0)< M,\quad\hbox{and}\quad
\max_{t\geq0}\alpha(B_b(t),t)< M\quad \hbox{for Problem 2,}
$$
then
$$
\max_{\Omega\cap\{t\leq T_0\}}(\alpha,\beta)(r,t)< M.
$$
\end{lemma}
\begin{remark}
By Lemmas \ref{lem1}-\ref{lem2}, when $1<\gamma<3$, $\{(\alpha,\beta)|\min(\alpha,\beta)\geq0, \max(\alpha,\beta)< M\}$is a domain invariant on time.
\end{remark}
\begin{proof}


We first rewrite the equations of $\alpha$ and $\beta$ in Lemma \ref{lemma_ric}, as
\[
\pp\alpha=-\frac{\gamma+1}{4}\alpha^2-\frac{3-\gamma}{4}\alpha\beta+A_2(\beta-\alpha)-(B_2-A_2)\alpha,
\]
and
\[
\pn\beta=-\frac{1+\gamma}{4}\beta^2-\frac{3-\gamma}{4}\alpha\beta+A_1(\alpha-\beta)-(B_1-A_1)\beta,
\]
when $r>0$ and $c_1c_2\neq 0$.

We first show two important relations:
\beq\label{BA12}
B_1-A_1\geq0,\quad B_2-A_2\geq 0.
\eeq
Once we proved \eqref{BA12}, to explain the idea,
we see that on the right boundary of domain $\{\min(\alpha,\beta)\geq0, \max(\alpha,\beta)< M\}$ on the $(\alpha,\beta)$-plane, $0\leq\beta\leq\alpha=M$. So $\pp\alpha\leq 0$, since $A_2\geq0$. Similarly, we know $\pn\beta\leq 0$ on the upper boundary of the invariant domain. By Lemma \ref{lem1}, this domain is invariant on time.

Now we prove \eqref{BA12}. It suggests by \eqref{inbo1} that
\begin{align}\label{bb1}
h\leq\frac{\gamma-1}{2}u\leq(\gamma-1)C_0,
\end{align}
and
\begin{align}\label{bb2}
c_1=u-h\geq \frac{3-\gamma}{\gamma-1}h\geq0,\quad c_2=u+h\geq u,
\end{align}
on $r\geq b>0$. We rewrite the expressions of $A_1$ and $B_1$ in Lemma \ref{lemma_ric} as
\[
-A_1=\frac{m }{rc_1^2}\left(\frac{\gamma-1}{4}(-u^3-u^2h)+\frac{1}{2}(uh^2+h^3)\right),
\]
and
\[
B_1=\frac{m}{rc_1^2}\left\{\frac{\gamma-1}{4}(u^3-u^2h)+\frac{1}{2}(uh^2-h^3)+
\frac{hu{c_1}}{{c_2}}(h+\frac{\gamma-1}{2}u)\right\}.
\]
Hence one calculates
\begin{align*}
B_1-A_1
=&\frac{m}{rc_1^2}\left\{-\frac{\gamma-1}{2}u^2h+uh^2+
\frac{hu{c_1}}{{c_2}}(h+\frac{\gamma-1}{2}u)\right\}\\
=&\frac{m}{rc_1^2}\left\{-\frac{\gamma-1}{2}\frac{2u^2h^2}{c_2}+uh^2+
\frac{h^2u{c_1}}{{c_2}}\right\}\quad \hbox{(plug in $c_1=u-h$)}\\
=&(3-\gamma)\frac{m}{rc_1^2}\frac{u^2h^2}{c_2}
\geq0.
\end{align*}
Here we note by \eqref{bb2} that the term $\frac{u^2h^2}{c_1^2 c_2}$ is well-defined.

Similarly, the expressions of $A_2$ and $B_2$ in Lemma \ref{lemma_ric} can be rewritten as
\[
A_2=
\frac{m c_1}{2rc_2^2}(\frac{\gamma-1}{2}u^2-h^2)
=\frac{m}{rc_2^2}(\frac{\gamma-1}{4}u^3-\frac{\gamma-1}{4}u^2h-\frac{1}{2}uh^2+\frac{1}{2}h^3),
\]
and
\[
B_2=\frac{m}{rc_2^2}\big(\frac{\gamma-1}{4}u^3+\frac{1}{2}h^3+
\frac{\gamma-1}{4}u^2h+
\frac{1}{2}uh^2+\frac{hu{c_2}}{{c_1}}(h-\frac{\gamma-1}{2}u)\big),
\]
from which we have
\begin{align*}
B_2-A_2
=&\frac{m}{rc_2^2}\big(
\frac{\gamma-1}{2}u^2h+
uh^2+\frac{hu{c_2}}{{c_1}}(h-\frac{\gamma-1}{2}u)\big)\\
=&(3-\gamma)\frac{m}{rc_2^2}\frac{u^2h^2}{c_1}
\geq0.
\end{align*}

Then we prove $\max_{(r,t)\in\Omega}(\alpha,\beta)(r,t)< M$, for any time $t\in[0,T]$, by contradiction. As in the previous Lemma \ref{lem1}, we may assume that there exists
a characteristic triangle tip or a characteristic quadrangle tip at $(\bar r,T)$ such that $\alpha=M$ or $\beta=M$ at  $(\bar r,T)$, but  $0\leq\alpha<M$ and $0\leq\beta<M$ in the characteristic triangle or characteristic quadrangle.

Without loss of generality, we assume $0\leq\alpha<M$ and $\beta=M$ at $(\bar r,T)$.
It is observed that
\begin{align*}
    \partial_1\beta< -\frac{1+\gamma}{4}\beta^2-\frac{3-\gamma}{4}\alpha\beta<0,
\end{align*}
at the vertex $(\bar r,T)$, which violates our assumption that $\beta<M$ in the characteristic triangle or characteristic quadrangle. We can find a similar contradiction when $0\leq\beta<M$ and $\alpha=M$ at $(\bar r,T)$.
Hence, we prove the lemma.
\end{proof}

We now derive a time-dependent density positive lower bound. A density lower bound was provided in Theorems 4.1 and 4.2 in \cite{CCW} by studying the solution in the Lagrangian coordinates.
Here, we give a much better and cleaner lower bound on density using our new Riccati system directly in the Euler coordinates and Lemma \ref{lem2}.

\begin{lemma}\label{lem3}
Let the assumptions in Lemma \ref{lem2} hold.
Moreover, suppose that the initial data satisfy
\beq\label{aa3}
\min_{r\in[b,\infty)}\rho(r,0)>0,\quad\hbox{and}\quad
\min_{t\in[0,T_0]}\rho(B_b(t),t)>0\quad \hbox{for Problem 2.}
\eeq
Then, for $1<\gamma<3$, the smooth solution on $t\in[0,T_0]$ for Problem 1 and 2 satisfies
\begin{align}\label{aa4}
\rho(r,t)\geq \bigg(\frac{b}{b+2C_0t}\bigg)^m\bar{\rho} e^{-M_bt},
\end{align}
where $M_b$ is a positive constant depending on $b$, and
\begin{align*}
\bar{\rho}=
\left\{
\begin{array}{l}
\dspt \min_{r\in[b,\infty)}\rho(r,0) \quad {\rm for\ Problem\ 1}, \\
\dspt \min\{\min_{r\in[b,\infty)}\rho(r,0), \min_{t\in[0,T_0]}\rho(B_b(t),t)\} \quad {\rm for\ Problem\ 2}.
\end{array}
\right.
\end{align*}
\end{lemma}
\begin{proof}
We set $\partial_0=\partial_t+u\partial_r$ and apply the definitions of $(\partial_1, \partial_2)$ to obtain
\beq\label{aa5}
\partial_0=\frac{\partial_1+\partial_2}{2},
\eeq
which together with \eqref{p1h} and \eqref{p2h} gives
\beq\label{aa6}
\partial_0h=\frac{\partial_1h+\partial_2h}{2}=-\frac{\gamma-1}{4}h\bigg(\alpha+\beta +\frac{mu^2}{rc_1} +\frac{mu^2}{rc_2}\bigg).
\eeq
By the relationship between $\rho$ and $h$,  we calculate
\begin{align}
\partial_0\bigg(\frac{1}{r^m\rho}\bigg)=&\frac{1}{r^m\rho}\bigg(-\frac{2}{\gamma-1}\frac{1}{h}\partial_0h -\frac{m}{r}\partial_0r\bigg) \nonumber \\
=&\frac{1}{r^m\rho}\bigg(\frac{\alpha+\beta}{2} +\frac{mu^2}{2rc_1} +\frac{mu^2}{2rc_2} -\frac{mu}{r}\bigg) \nonumber \\
=&\frac{1}{r^m\rho}\bigg(\frac{\alpha+\beta}{2} +\frac{mh^2}{2rc_1} +\frac{mh^2}{2rc_2} \bigg),
\end{align}
from which one has
\begin{align}\label{aa7}
\partial_0\ln\bigg(\frac{1}{r^m\rho}\bigg)=\frac{\alpha+\beta}{2} +\frac{mh^2}{2rc_1} +\frac{mh^2}{2rc_2}.
\end{align}
Recalling \eqref{bb1} and \eqref{bb2}, we have
\begin{align}\label{aa8}
\frac{mh^2}{2rc_1}, \frac{mh^2}{2rc_2}\leq \frac{m(\gamma-1)}{2b(3-\gamma)}h\leq  \frac{m(\gamma-1)}{b(3-\gamma)}C_0.
\end{align}
Here we used the fact $r\geq b$ for any point $(r,t)$ in $\Omega$.

Let $r=r_0(t)=r_0(t;\xi,\eta)$ be the curve defined by
\begin{align}
\frac{d r_0(t)}{d t}=u(r_0(t),t),\quad r_0(\xi)=\eta,
\end{align}
for any point $(\xi,\eta)\in\Omega$ and $t\leq \xi$. We denote the intersection point of $r=r_0(t)$ and the initial line $t=0$ or boundary curve $B_b(t)$ by $(r_0(t_0), t_0)$. Then
\begin{align}
r_0(t_0)\leq r_0(t)\leq r_0(t_0) +\int_{t_0}^tu(r_0(\tau),\tau)\ d\tau \leq r_0(t_0) +2C_0t,
\end{align}
which indicates that
\begin{align}\label{aa9}
\frac{r_0(t_0)}{r_0(t)}\geq \bigg(1+\frac{2C_0t}{r_0(t_0)}\bigg)^{-1} \geq \frac{b}{b+2C_0t}.
\end{align}
Integrating \eqref{aa7} along $r_0(t)$ and applying Lemma \ref{lem2}, \eqref{aa8} and \eqref{aa9}, we conclude that
\begin{align}\label{aa10}
\rho(r,t)\geq \bigg(\frac{r_0(t_0)}{r}\bigg)^m\rho_0 e^{-M_bt} \geq  \bar{\rho}\bigg(\frac{b}{b+2C_0t}\bigg)^m e^{-M_bt},
\end{align}
where $\rho_0=\rho(r_0(t_0),t_0)$, and
\begin{align}\label{aa10a}
M_b=M+\frac{2m(\gamma-1)}{b(3-\gamma)}C_0.
\end{align}
The proof of the lemma is finished.
\end{proof}

\section{Global existence result on the entire half line $r\in[0,\infty)$}\label{S5}

The first goal of this section is to show that the existence of $C^1$ solution in $\Omega$ in Problem 1 and 2, when the initial data include only rarefaction, i.e. $\min(\alpha,\beta)(r,0)\geq0$ when $r>b>0$. For Problem 2, we also assume that there is only rarefaction on the boundary $r=B_b(t)$, i.e. $\min\alpha(B_b(t),t)\geq0$ for $t\geq0$.

Then, letting $b\rightarrow 0$, we can construct the global existence result on $r\geq0, t\geq 0$, with $\rho(0,t)=u(0,t)=0$ at the origin.

We will also construct some global classical solutions with positive density, by adding some special initial data on $r\in[0,b]$.

To prove these results, we need to use the $L^\infty$ bound in \eqref{inbo1}, the density lower bound in \eqref{aa4} and $C^1$ bound for $w$ and $z$.

\subsection{Global existence for Problem 1 and 2}

For Problem 1, we have the following existence theorem
\begin{theorem}\label{ex 1}
Let the initial data $(\rho_0(r), u_0(r))\in C^1([b,\infty))$ satisfy Assumption \ref{asu_1} and $\min_{r\in[b,\infty)}\rho_0(r)>0$. Suppose that
\begin{align}\label{aa13}
\min_{r\in[b,\infty)}\{\alpha_0(r), \beta_0(r)\}\geq0,\quad  \max_{r\in[b,\infty)}\{\alpha_0(r), \beta_0(r)\}<M,
\end{align}
with
\begin{align}\label{aa14}
\begin{array}{l}
\dspt \alpha_0(r)=u_{0}'(r) +\frac{2}{\gamma-1}h_{0}'(r) +\frac{m}{r}\frac{h_0(r)u_0(r)}{u_0(r)+h_0(r)}, \\[8pt]
\dspt \beta_0(r)=u_{0}'(r) -\frac{2}{\gamma-1}h_{0}'(r) -\frac{m}{r}\frac{h_0(r)u_0(r)}{u_0(r)-h_0(r)}.
\end{array}
\end{align}
Here $M$ is a positive constant and $h_0(r)=\sqrt{K\gamma}\rho_0(r)^{(\gamma-1)/2}$. Then, for $1<\gamma<3$, Problem 1 admits a global $C^1$ solution on domain of dependence $\Omega$ on the $(r,t)$-plane with base $t=0$ and $(b,\infty)$ for any $b>0$. Moreover, the solution $(\rho,u)(r,t)$ satisfies
\begin{align}\label{aa15}
\frac{2\sqrt{K\gamma}}{\gamma-1}\rho^{\frac{\gamma-1}{2}}(r,t)\leq u(r,t)\leq 2C_0,\
\rho(r,t)\geq \bigg(\frac{b}{b+2C_0t}\bigg)^m\bar{\rho} e^{-M_bt},
\end{align}
for some positive constant $M_b$, and
\begin{align}\label{aa16}
\min_{\Omega}(\alpha,\beta)(r,t)\geq0,\quad \max_{\Omega}(\alpha,\beta)(r,t)<M.
\end{align}
Here $\bar{\rho}=\min_{r\in[b,\infty)}\rho_0(r)$.
\end{theorem}
\begin{proof}
The proof of the theorem is based on the classical framework of Li \cite{LiBook} by extending the local smooth solution to global domain. The local existence of smooth solution to Problem 1 follows from the classical results, see, e.g. Li and Yu \cite{Li-Yu} or Bressan \cite{Bressan}. In order to extend the local solution to a global one, it suffices to establish the a priori $C^1$ estimates of the solution on the domain $\Omega$. Actually, by examining the proof process of local classical solution in \cite{Li-Yu}, the
existence time $\delta$ of the smooth solution depends only on the norm $\|\Gamma^*\|$ and the $C^1$ norm $\|(w_0(r), z_0(r))\|_{C^1([b,\infty))}$, where $\Gamma^*$ is the following set of functions
\begin{align}\label{aa17}
\Gamma^*=\bigg\{&c_1, c_2, \frac{1}{c_2-c_1}, \frac{\partial c_1}{\partial w}, \frac{\partial c_1}{\partial z}, \frac{\partial c_2}{\partial w}, \frac{\partial c_2}{\partial z},  \nonumber \\
&\frac{uh}{r}, \partial_r\bigg(\frac{uh}{r}\bigg),  \partial_w\bigg(\frac{uh}{r}\bigg), \partial_z\bigg(\frac{uh}{r}\bigg)\bigg\},
\end{align}
and $w_0(r)=u_0(r)+\frac{2}{\gamma-1}h_0(r)$, $z_0(r)=u_0(r)-\frac{2}{\gamma-1}h_0(r)$. See Remark 4.1. in Chapter 1 in \cite{Li-Yu}. From the above and the a priori estimates in Section 4, we know that, for any number $b>0$ and any time $T>0$, the existence time $\delta$ depends only on $b$ and $T$, which means that, for the fixed $b$ and $T$, the local existence time $\delta$ is a constant. Hence we can solve a finite number of local existence problems to extend the solution in the region $\Omega\cap\{0\leq t\leq \delta\}$ to the global region $\Omega\cap\{0\leq t\leq T\}$. By the arbitrariness of $T$, we can obtain the smooth solution on the global region $\Omega$ for any fixed $b>0$. The properties of solution in \eqref{aa15} and \eqref{aa16} can be acquired directly by the results in Section 4.
\end{proof}

For Problem 2, we have
\begin{theorem}\label{ex 2}
Let the initial conditions of $(\rho_0(r),u_0(r))$ in Theorem \ref{ex 1} hold. Let $r=B_b(t) (t\in[0,\infty))$ be a smooth increasing curve. Assume that the boundary value $(\rho_b(t), u_b(t))\in C^1([0,\infty))$ on $B_b(t)$ satisfies Assumption \eqref{asu_2} and
\begin{align}\label{aa18}
\left\{
\begin{array}{l}
\frac{d B_b(t)}{d t}=u_b(t)-h_b(t), \\
B_b(0)=b>0,
\end{array}
\right.
\end{align}
where $h_b(t)=\sqrt{K\gamma}\rho_b(t)^{\frac{\gamma-1}{2}}$ with $\min_{t\in[0,T]}\rho_b(t)>0$ for any $T>0$.
Moreover, the function $\alpha_b(t)$ satisfies
\begin{align}\label{aa19}
\min_{t\geq0}\alpha_b(t)\geq0,\quad  \max_{t\geq0}\alpha_b(t)<M,
\end{align}
where
\begin{align}\label{aa20}
\alpha_b(t)=-\frac{1}{B_{b}^m(t)\rho_b(t)(u_b(t)-h_b(t))}\frac{d(B_{b}^m(t)\rho_b(t)u_b(t))}{d t},
\end{align}
and $M$ is a positive constant. In addition, suppose that the following compatibility conditions hold
\begin{align}\label{aa21}
\begin{array}{c}
\dspt (\rho_0(b),u_0(b))=(\rho_b(0),u_b(0)),\ u_{0}'(b)-\frac{2}{\gamma-1}h_{0}'(b)=\frac{mu_b(0)h_b(0)}{b}, \\[8pt]
\dspt \frac{d}{d t}\bigg(u_b(t)-\frac{2}{\gamma-1}h_b(t)\bigg) =\frac{mu_b(t)h_b(t)}{B_b(t)}.
\end{array}
\end{align}
Then, for $1<\gamma<3$, Problem 2 admits a global $C^1$ solution on the domain $\Omega$ bounded by the initial line $t=0$ and the curve $r=B_b(t)$ for any $b>0$. Moreover, the solution $(\rho,u)(r,t)$ satisfies
\begin{align}\label{aa22}
\frac{2\sqrt{K\gamma}}{\gamma-1}\rho^{\frac{\gamma-1}{2}}(r,t)\leq u(r,t)\leq 2C_0,\
\rho(r,t)\geq \bigg(\frac{b}{b+2C_0t}\bigg)^m\bar{\rho} e^{-M_bt},
\end{align}
for some positive constant $M_b$, and
\begin{align}\label{aa23}
\min_{\Omega}(\alpha,\beta)(r,t)\geq0,\quad \max_{\Omega}(\alpha,\beta)(r,t)<M.
\end{align}
Here $\bar{\rho}=\min\{\min_{r\in[b,\infty)}\rho_0(r), \min_{t\geq0}\rho_b(t)\}$.
\end{theorem}
\begin{proof}
Based on Theorem \ref{ex 1}, it suffices to consider the global existence of the Goursat type boundary value problem on the domain $\Omega_0$ which is bounded by the 1-characteristic $r=B_b(t)$ and the 2-characteristic $r=C_b(t)$ starting from point $(b,0)$. Here the curve $r=C_b(t)$ is determined in solving Problem 1. Moreover, the boundary data on $r=C_b(t)$ satisfies \eqref{aa15} and \eqref{aa16}. For this typical Goursat problem, we can also utilize the classical framework of Li \cite{LiBook} to extend the local smooth solution to global domain. The local existence time $\delta$ of the smooth solution in Li and Yu \cite{Li-Yu}
depends only on the norm $\|\Gamma^*\|$ and the $C^1$ norms of $(w,z)$ on the boundaries $r=B_b(t)$ and $r=C_b(t)$. Thanks to the a priori estimates in Section 4, the existence time $\delta$ still depends only on the number $b>0$ and the time $T>0$. By solving a finite number of local Goursat and Cauchy problems, we can acquire the global existence of smooth solution on the domain $\Omega_0\cap\{0\leq t\leq T\}$.
Due to the arbitrariness of $T$, we obtain the smooth solution over the entire region $\Omega_0$, which completes the proof of the theorem.
\end{proof}

\subsection{Global existence on the entire half line $r\in[0,\infty)$}

Let's start to construct the global solution on the entire half line $r\in[0,\infty)$.
We further assume that the boundary condition $u(0,t)=0$ at the origin, which is a reasonable
physical assumption.

For the data of the density at the origin, a direct and simple choice is
$\rho(0,t)=0$, which makes that the initial assumption \eqref{inbo0} is satisfied on the entire half line $r\in[0,\infty)$. For this choice, one can apply the results in Theorem \ref{ex 1} to construct the global smooth solution. Suppose that the initial conditions in Theorem \ref{ex 1} hold for any $b>0$. Moreover, we assume that, when $r$ approaches 0, the limits of $\alpha_0(r)$ and $\beta_0(r)$ exist and are nonnegative.
For any point $(r,t)$ with $r>0$, we let $b$ small enough such that $(r,t)$ is in the corresponding domain $\Omega$. Here we used the result in \cite{CCW} that the characteristic starting from $(r,t)=(0,0)$
must be $r=0$. Hence we obtain the smooth solution on the entire domain $r\geq0, t\geq0$.

Thus we have
\begin{theorem}\label{ex 3_0}
Let the initial data $(\rho_0(r), u_0(r))\in C^1([0,\infty))$ satisfy Assumption \ref{asu_1} and $\rho_0(0)=0$, $u_0(0)=0$, $\rho_0(r)>0$ for $r>0$. Suppose that
\begin{align}\label{cc1}
\min_{r\in[0,\infty)}\{\alpha_0(r), \beta_0(r)\}\geq0,\quad  \max_{r\in[0,\infty)}\{\alpha_0(r), \beta_0(r)\}\leq M,
\end{align}
where $\alpha_0(r)$ and $\beta_0(r)$ are given in \eqref{aa14}, and $M$ is a positive constant. Here $\alpha_0(0)=\lim_{r\rightarrow0^+}\alpha_0(r)$ and $\beta_0(0)=\lim_{r\rightarrow0^+}\beta_0(r)$ which are assumed to exist. Then, for $1<\gamma<3$, the radially symmetric Euler equations \eqref{Euler1}-\eqref{Euler3} admit a global $C^1$ solution $(\rho,u)(r,t)$ on the entire domain $r\geq0, t\geq0$. Moreover, the solution
satisfies $\rho(0,t)=0$, $u(0,t)=0$ and
\begin{align}\label{cc2}
\frac{2\sqrt{K\gamma}}{\gamma-1}\rho^{\frac{\gamma-1}{2}}(r,t)\leq u(r,t)\leq 2C_0,\quad
\rho(r,t)>0,\quad \forall\ r>0, t\geq0
\end{align}
and
\begin{align}\label{cc3}
\min_{r\geq0, t\geq0}(\alpha,\beta)(r,t)\geq0,\quad \max_{r\geq0, t\geq0}(\alpha,\beta)(r,t)\leq M.
\end{align}
\end{theorem}

There is another choice that the density $\rho$ is positive at the origin. Inspired in \cite{sideris1} and \cite{CCW}, we may construct a smooth solution satisfying affine motion close to $r=0$ and then combine it with the smooth solution obtained in Problem 2 to acquire a global solution on the entire domain $r\geq0, t\geq0$.

In order to proceed, we start from the affine motions
\begin{align}
r(y, t)=a(t) y.
\end{align}
In material coordinates, the velocity associated to this motion is
\begin{align}\label{bb3}
u(r(y, t), t)=\frac{d}{dt}r(y, t)=a'(t)y,
\end{align}
from which we get the material time derivative of the
velocity
\begin{align}\label{aa24}
D_tu(r(y, t),t)=a''(t)y.
\end{align}
We take the material derivative on $\rho$ and employ the equation \eqref{Euler1} to achieve
\begin{align}
\frac{d}{dt}\rho(r(y, t), t)=&\frac{d}{dt}\rho(a(t)y, t) =\partial_t\rho+u\partial_r\rho \nonumber \\
=&-\frac{m}{r}\rho u-\rho\partial_r u=-(m+1)\rho \frac{a'(t)}{a(t)}.
\end{align}
Thus there holds
\begin{align}
\frac{d}{dt}\left(\ln(a^{m+1}\rho)\right)=0,
\end{align}
which leads to
\begin{align}\label{aa25}
\rho(r(y, t), t)=\rho(a(t)y, t)=\left(\frac{a(0)}{a(t)}\right)^{m+1}\rho_{0A}(y),
\end{align}
where $\rho_{0A}(y)$ is the initial value. Notice that the equation \eqref{Euler2} can be rewritten as
\begin{align}
D_tu+\frac{\gamma K}{\gamma-1}\partial_r\rho^{\gamma-1}=0,
\end{align}
which together with \eqref{aa24} and \eqref{aa25} yields
\begin{align}
a''(t)y+\frac{\gamma K}{\gamma-1}\left(\frac{a(0)}{a(t)}\right)^{(m+1)(\gamma-1)} \cdot\frac{1}{a}\partial_y(\rho_{0A}^{\gamma-1}(y))=0.
\end{align}
Hence we could separate the equations for $a(t)$ and $\rho_0(y)$ as
\begin{align}\label{aa26}
a''(t)=\frac{1}{a(t)}\left(\frac{a(0)}{a(t)}\right)^{(m+1)(\gamma-1)},\quad
y=-\frac{\gamma K}{\gamma-1}\partial_y\big(\rho_{0A}^{\gamma-1}(y)\big).
\end{align}
It follows directly by the second equation of \eqref{aa26} that
\begin{align}\label{aa27}
\rho_{0A}(y)=\left(\rho_c^{\gamma-1}-\frac{\gamma-1}{2\gamma K}y^2\right)^\frac{1}{\gamma-1},
\end{align}
where $ \rho_{0A}(0)=\rho_c$ is a positive constant. Moreover, for affine motion $a$, one has $r(y, 0)=y$, which implies that $a(0)=1$. Thus we could have the ODE problem
\begin{align}\label{aa28}
a''(t)=a^{-(m+1)(\gamma-1)-1},\quad
a(0)=1,\quad a'(0)=v_a,
\end{align}
where $v_a>0$ is the initial velocity. This problem equals to
\begin{align}\label{aa29}
a'(t)=\left(\frac{2}{d(\gamma-1)}(1-a^{-(m+1)(\gamma-1)})+(v_a)^2\right)^{\frac{1}{2}},\quad
a(0)=1.
\end{align}
It is obvious that for $t>0$
\begin{align}\label{aa30}
0<v_a<a'(t)\leq\left(\frac{2}{(m+1)(\gamma-1)}+(v_a)^2\right)^{\frac{1}{2}},\quad a(t)>1,
\end{align}
By the standard ODE theory, we know that there exists a global unique solution for problem \eqref{aa29}, which equivalently gives a global unique solution for problem \eqref{aa28}.

Let's construct the central affine solution. For any fixed $\rho_c>0$ and $v_a>0$, we know by \eqref{bb3} and \eqref{aa27} that
\begin{align}\label{cc4}
\frac{2\sqrt{K\gamma}}{\gamma-1}\rho_{0A}^{\frac{\gamma-1}{2}}(y) =\frac{2\sqrt{K\gamma}}{\gamma-1}(\rho_c^{\gamma-1}-\frac{\gamma-1}{2\gamma K}y^2)^\frac{1}{2}\quad\mbox{and}\quad u_0=v_a y,
\end{align}
which implies there exists  $b>0$ such that
\begin{align}\label{aa31}
\frac{2\sqrt{K\gamma}}{\gamma-1}\rho_{0A}^{\frac{\gamma-1}{2}}(b)\leq u_0(b).
\end{align}
Thus we consider the initial data that satisfies the following assumption
\begin{asu}\label{asu_3}
Assume that the initial data $(\rho_0,u_0)(r)\in C^1([0,\infty))$ satisfy that $(\rho_0, u_0)(r)=(\rho_{0A}(r), v_a r)$ when $0\leq r\leq b$, and (\ref{inbo0}) holds when $r>b$.
\end{asu}
Denote the domain of dependence with base $r\in[0, b]$ by $\Pi$. We know that the right boundary of $\Pi$ is the $1$-characteristic starting from $(b,0)$, that is $r=B_b(t)$. It is easy to see by \eqref{aa31} that this boundary is moving outward, \textit{i.e.} away from the origin, as $t$ increases. Inside the region $\Pi$, we have
\begin{align}\label{aa32}
(\rho, u)(r,t)= \left(\frac{\rho_{0A}\big(\frac{r}{a(t)}\big)}{a(t)^{m+1}}, \frac{a'(t)}{a(t)}r\right).
\end{align}
which together with \eqref{aa27} and \eqref{aa30} leads to the uniform $L^\infty$ upper bound of $(\rho,u)$
\begin{align}\label{aa33}
\rho\leq \rho_c, \qquad u<\left(\frac{2}{(m+1)(\gamma-1)}+(v_a)^2\right)^{\frac{1}{2}}b.
\end{align}

Now we seek suitable conditions such that the above constructed affine solution satisfy the boundary condition of Theorem \ref{ex 2} on the $1$-characteristic $r=B_b(t)$.
By the expression of the affine solution in \eqref{aa32}, one has
\begin{align}\label{add34}
h(r,t)=\sqrt{K\gamma}\rho^{\frac{\gamma-1}{2}} =\sqrt{K\gamma}a^{-\frac{(m+1)(\gamma-1)}{2}}\sqrt{\rho_{c}^{\gamma-1}-\frac{\gamma-1}{2\gamma K}\frac{r^2}{a^2}},
\end{align}
for $y=\frac{r}{a}\in[0,b]$,
from which we obtain
\begin{align}\label{aa35}
\begin{array}{l}
\displaystyle w=\frac{a'}{a}r +\frac{2\sqrt{K\gamma}}{\gamma-1} a^{-\frac{(m+1)(\gamma-1)}{2}}\sqrt{\rho_{c}^{\gamma-1}-\frac{\gamma-1}{2\gamma K}\frac{r^2}{a^2}},\\[12pt]
\displaystyle z=\frac{a'}{a}r -\frac{2\sqrt{K\gamma}}{\gamma-1} a^{-\frac{(m+1)(\gamma-1)}{2}}\sqrt{\rho_{c}^{\gamma-1}-\frac{\gamma-1}{2\gamma K}\frac{r^2}{a^2}}.
\end{array}
\end{align}
It is easy to see that $w^2\geq z^2$, which together with the equation of $z$ \eqref{z eq}
yields that $z$ is increasing along the curve $B_b(t)$. Thus
$$
z(B_b(t),t)> z(b,0).
$$
In view of \eqref{aa35} and the properties of $a(t)$, if
\begin{align}\label{aa36}
z(b,0)=v_a b -\frac{2\sqrt{K\gamma}}{\gamma-1} \sqrt{\rho_{c}^{\gamma-1}-\frac{\gamma-1}{2\gamma K}b_{A}^2}\geq0,
\end{align}
one then has $z(B_b(t),t)>0$ and then $c_1=u-h>0$ on $r=B_b(t)$ by $1<\gamma< 3$.
We next differentiate $w$ and $z$ in \eqref{aa35} with respect to $r$ to acquire
\begin{align}\label{aa37}
\begin{array}{l}
\displaystyle w_r=\frac{a'}{a}-(K\gamma)^{-\frac{1}{2}}a^{-\frac{(m+1)(\gamma-1)+4}{2}}\bigg(\rho_{c}^{\gamma-1} -\frac{\gamma-1}{2K\gamma}\frac{r^2}{a^2}\bigg)^{-\frac{1}{2}}r,\\[12pt]
\displaystyle
z_r=\frac{a'}{a}+(K\gamma)^{-\frac{1}{2}}a^{-\frac{(m+1)(\gamma-1)+4}{2}}\bigg(\rho_{c}^{\gamma-1} -\frac{\gamma-1}{2K\gamma}\frac{r^2}{a^2}\bigg)^{-\frac{1}{2}}r.
\end{array}
\end{align}
According to the expressions of $\alpha$ and $\beta$ in \eqref{def alpha} and \eqref{def beta},
we find that, if $w_r\geq0$ on $r=B_b(t)$, then there holds $\alpha>0$. To get it, by the expression of $w_r$ in \eqref{aa37} and the relation $r=a(t)y$, it suffices to
\begin{align}\label{aa38}
a^{\frac{(m+1)(\gamma-1)}{2}}a'(t)\geq (K\gamma)^{-\frac{1}{2}}\bigg(\rho_{c}^{\gamma-1} -\frac{\gamma-1}{2K\gamma}y^2\bigg)^{-\frac{1}{2}}y.
\end{align}
Due to the monotonic increasing property of $y$ on the left-side of inequality \eqref{aa38}, we only need
\begin{align}\label{aa39}
a^{\frac{(m+1)(\gamma-1)}{2}}a'(t)\geq (K\gamma)^{-\frac{1}{2}}\bigg(\rho_{c}^{\gamma-1} -\frac{\gamma-1}{2K\gamma}b^2\bigg)^{-\frac{1}{2}}b,
\end{align}
by $y\in[0,b]$. On the other hand, one recalls the properties of $a(t)$ to obtain
$$
a(t)\geq1,\quad a'(t)\geq v_a.
$$
Hence to ensure $\alpha>0$ on $r=B_b(t)$, we only need
\begin{align}\label{aa40}
v_a\geq (K\gamma)^{-\frac{1}{2}}\bigg(\rho_{c}^{\gamma-1} -\frac{\gamma-1}{2K\gamma}b^2\bigg)^{-\frac{1}{2}}b.
\end{align}
For the value $\beta(b,0)$, we calculate by \eqref{aa37} and the expression of $\beta$
\begin{align}\label{aa41}
\beta(b,0)=&v_a+(K\gamma)^{-\frac{1}{2}}\bigg(\rho_{c}^{\gamma-1} -\frac{\gamma-1}{2K\gamma}b^2\bigg)^{-\frac{1}{2}}b \nonumber \\
&-\frac{m}{b}\cdot\frac{v_ab\cdot \sqrt{K\gamma}\sqrt{\rho_{c}^{\gamma-1}-\frac{\gamma-1}{2K\gamma}b^2}}{v_ab-\sqrt{K\gamma} \sqrt{\rho_{c}^{\gamma-1}-\frac{\gamma-1}{2K\gamma}b^2}} \nonumber \\
>& v_a-\frac{mv_a \sqrt{K\gamma}\sqrt{\rho_{c}^{\gamma-1}-\frac{\gamma-1}{2K\gamma}b^2}}{v_ab-\sqrt{K\gamma} \sqrt{\rho_{c}^{\gamma-1}-\frac{\gamma-1}{2K\gamma}b^2}} \nonumber \\
=& \frac{v_a \big[ v_ab-(m+1) \sqrt{K\gamma}\sqrt{\rho_{c}^{\gamma-1}-\frac{\gamma-1}{2K\gamma}b^2}\big]}{v_ab-\sqrt{K\gamma} \sqrt{\rho_{c}^{\gamma-1}-\frac{\gamma-1}{2K\gamma}b^2}},
\end{align}
from which one has $\beta(b,0)\geq0$ provided that
\begin{align}\label{aa42}
v_ab\geq(m+1) \sqrt{K\gamma}\sqrt{\rho_{c}^{\gamma-1}-\frac{\gamma-1}{2K\gamma}b^2}.
\end{align}
Summing up \eqref{aa36}, \eqref{aa40} and \eqref{aa42}, we see that if the parameters $v_a, \rho_c$ and $b$ satisfy
\begin{align}\label{aa43}
\rho_{c}^{\gamma-1}>\frac{\gamma-1}{2K\gamma}b^2,
\end{align}
and
\begin{align}\label{aa44}
v_a   \geq & \max\bigg\{\frac{(m+1)\sqrt{K\gamma}}{b} \sqrt{\rho_{c}^{\gamma-1}-\frac{\gamma-1}{2K\gamma}b^2}, \nonumber \\   & \frac{2\sqrt{K\gamma}}{(\gamma-1)b} \sqrt{\rho_{c}^{\gamma-1}-\frac{\gamma-1}{2K\gamma}b^2}, \ \ \frac{b}{\sqrt{K\gamma}\sqrt{\rho_{c}^{\gamma-1}-\frac{\gamma-1}{2K\gamma}b^2}}\bigg\},
\end{align}
then there hold
\begin{align}\label{aa45}
\begin{array}{c}
\alpha(B_b(t),t)>0,\quad \beta(b,0)\geq0,\\
w(B_b(t),t)>z(B_b(t),t)\geq0,\quad c_1(B_b(t),t)>0.
\end{array}
\end{align}
Note that the function $z$ is strictly monotonically increasing along $\Gamma$, we can obtain $z(B_b(t),t)>0$ for $t>0$. Moreover, one recall the equation of $\beta$ in \eqref{beta_eq} with $A_1>0$ on $r=B_b(t)$ to easily achieve $\beta(B_b(t),t)>0$. We next derive the upper bound of $(\alpha, \beta)(B_b(t),t)$. It suggests by \eqref{aa30} and \eqref{aa37} that
\begin{align}\label{aa46}
w_r(B_b(t),t)\leq \frac{a'(t)}{a(t)}\leq \left(\frac{2}{(m+1)(\gamma-1)}+(v_a)^2\right)^{\frac{1}{2}},
\end{align}
which together with \eqref{aa33} and \eqref{def alpha} arrives at
\begin{align}\label{aa47}
\alpha(B_b(t),t)= w_r(B_b(t),t) +\frac{mh(B_b(t),t)u(B_b(t),t)}{B_b(t)c_2(B_b(t),t)}\leq M,
\end{align}
where the positive constant $M$ depends only on $v_a$. In addition, by the construction of the affine solution in \eqref{aa32}, it is easy to check that all compatibility conditions in \eqref{aa21} are satisfied. Hence, all conditions of Theorem \ref{ex 2} hold provided that the parameters $v_a, \rho_c$ and $b$ satisfy \eqref{aa43} and \eqref{aa44}.

Therefore, we have
\begin{theorem}\label{ex 3}
Let the initial data $(\rho_0(r), u_0(r))\in C^1([0,\infty))$ satisfy Assumption \ref{asu_3} and $\min_{r\in[0,\infty)}\rho_0(r)>0$.
Suppose that the positive parameters $v_a, \rho_c$ and $b$ satisfy \eqref{aa43} and \eqref{aa44}, where $b>0$, $\rho_c=\rho_0(0)$, $v_a=\frac{u_0(b)}{b}$. Then, for $1<\gamma<3$, the radially symmetric solution of Euler equations  \eqref{Euler1}-\eqref{Euler3} admit a global $C^1$ solution $(\rho,u)(r,t)$ on the entire domain $r\geq0, t\geq0$. Moreover, the solution takes the form in \eqref{aa32} on the left-hand region of the 1-characteristic starting from $(b,0)$, and satisfies \eqref{aa22} and \eqref{aa23} on its right-hand region.
\end{theorem}
\begin{proof}
Based on the above analysis and the result in Theorem \ref{ex 2}, we just need to check the existence of parameters $v_a, \rho_c$ and $b$ such that \eqref{aa43} and \eqref{aa44} can be satisfied simultaneously.
Indeed, for any fixed positive constant $\rho_c$, we can freely choose the parameter $b$ as long as the following relationship holds
$$
0<b<\sqrt{\frac{2K\gamma}{\gamma-1}\rho_{c}^{\gamma-1}}.
$$
Thus the constants $\rho_c$ and $b$ satisfy
$$
\rho_{c}^{\gamma-1}>\frac{\gamma-1}{2K\gamma}b^2,
$$
which is \eqref{aa43}. Now the three terms in the right-hand side of \eqref{aa44} are fixed constants. Then we choose $v_a$ sufficiently large such that it is greater than these three constants. Hence \eqref{aa44} is fulfilled. Furthermore, the inequality \eqref{aa31} follows directly from \eqref{aa44}. The proof of the theorem is complete.
\end{proof}
\begin{remark}
As we just showed, there is a broad class of positive parameters $v_a, \rho_c$ and $b$ satisfying the conditions in \eqref{aa43} and \eqref{aa44}. This means that there are many affine solutions that meet the requirements.
\end{remark}

\section{Singularity formation}\label{S6}

To prove the singularity formation when the initial data include strong compression somewhere, we need to further decouple the system on $\alpha$ and $\beta$.
In order to do that, we first show the following lemma.
\begin{lemma}\label{eqtilde}
Let $\lambda\geq0$ be any real number. For smooth solution of \eqref{Euler1}-\eqref{Euler3}, we have the following equations on the weighted variables $h^{-\lambda}\alpha$ and $h^{-\lambda}\beta$ 
\begin{align}
\partial_1\big(h^{-\lambda}\beta\big) = &-\frac{1+\gamma}{4}h^{\lambda} \big(h^{-\lambda}\beta\big)^2
      +\Big(\frac{\gamma-3}{4} +\frac{\gamma-1}{2}\lambda\Big)h^{\lambda}
       \big(h^{-\lambda}\alpha\big)\big(h^{-\lambda}\beta\big) \nonumber \\
&+A_1 \big(h^{-\lambda}\alpha\big) -
       B_1 \big(h^{-\lambda}\beta\big)
       +\frac{\gamma-1}{2}\lambda\frac{m}{rc_2}u^2
      \big(h^{-\lambda}\beta\big), \label{dd2} \\
\partial_2\big(h^{-\lambda}\alpha\big) =&-\frac{1+\gamma}{4}h^{\lambda} \big(h^{-\lambda}\alpha\big)^2
      +\Big(\frac{\gamma-3}{4} +\frac{\gamma-1}{2}\lambda\Big)h^{\lambda}
       \big(h^{-\lambda}\alpha\big)\big(h^{-\lambda}\beta\big) \nonumber \\
&+A_2 \big(h^{-\lambda}\beta\big) -
       B_2 \big(h^{-\lambda}\alpha\big)
       +\frac{\gamma-1}{2}\lambda\frac{m}{rc_1}u^2
      \big(h^{-\lambda}\alpha\big). \label{dd3}
\end{align}
\end{lemma}
\begin{proof}
By direct calculations, one has
\begin{align}\label{aa11}
\partial_1\beta=& \partial_1(h^{\lambda}\cdot h^{-\lambda}\beta)= (h^{-\lambda}\beta)\partial_1(h^{\lambda}) +
    h^{\lambda}\partial_1(h^{-\lambda}\beta)\nonumber\\
=& (h^{-\lambda}\beta)\lambda h^{\lambda-1}\partial_1(h)+h^{\lambda}\partial_1(h^{-\lambda}\beta)\nonumber\\
= & (h^{-\lambda}\beta)\lambda h^{\lambda-1}(-\frac{\gamma-1}{2}h\alpha-\frac{\gamma-1}{2}\frac{m}{rc_2}u^2h) +h^{\lambda}\partial_1(h^{-\lambda}\beta)\nonumber\\
=& h^{\lambda}\partial_1(h^{-\lambda}\beta)
      - \frac{\gamma-1}{2}\lambda\frac{m}{rc_2}u^2h^{\lambda}
      (h^{-\lambda}\beta) \nonumber \\
&-\frac{\gamma-1}{2}\lambda h^{2\lambda}(h^{-\lambda}\alpha)(h^{-\lambda}\beta).
\end{align}
On the other hand, it follows by the equation of $\beta$ in Lemma \ref{lemma_ric} that
\begin{align}\label{aa12}
\partial_1\beta
   =& -\frac{1+\gamma}{4}h^{2\lambda} (h^{-\lambda}\beta)^2
      -\frac{3-\gamma}{4}h^{2\lambda}
       (h^{-\lambda}\alpha)(h^{-\lambda}\beta) \nonumber \\
&+A_1 h^{\lambda}(h^{-\lambda}\alpha) - B_1 h^{\lambda}(h^{-\lambda}\beta).
\end{align}
Combining \eqref{aa11} and \eqref{aa12} and arranging the resulting gets the equation \eqref{dd2}. The equation \eqref{dd3} can be derived analogously.
\end{proof}

Specifically, we take $\lambda=\frac{3-\gamma}{2(\gamma-1)}$ in \eqref{dd2} and \eqref{dd3} and denote  
\begin{align}\label{dd17}
\tilde{\alpha} = h^{\frac{\gamma-3}{2(\gamma - 1)}}\alpha, \quad
\tilde{\beta} = h^{\frac{\gamma-3}{2(\gamma - 1)}}\beta,
\end{align}
to obtain
\begin{align}
\partial_1\tilde{\beta} = -\frac{1+\gamma}{4}h^{\frac{3-\gamma}{2(\gamma-1)}}\tilde{\beta}^2 +\frac{3-\gamma}{4}\frac{m}{r}\frac{u^2}{c_2}\tilde{\beta} -B_1\tilde{\beta}+A_1\tilde{\alpha}, \label{tildeab}\\
\partial_2\tilde{\alpha} = -\frac{1+\gamma}{4}h^{\frac{3-\gamma}{2(\gamma-1)}}\tilde{\alpha}^2 +\frac{3-\gamma}{4}\frac{m}{r}\frac{u^2}{c_1}\tilde{\alpha} -B_2\tilde{\alpha}+A_2\tilde{\beta}. \label{tildeab2}
\end{align}
We note that these two equations are decoupled in its leading quadratic order term.

To prove the desired singularity formation result,
one difficulty is to control the last term in \eqref{tildeab} and \eqref{tildeab2}. We cannot directly use the upper bound of $\alpha$ and $\beta$ in Lemma \ref{lem2} and the density lower bound in Lemma \ref{lem3}, since our initial $\alpha$ and $\beta$ might be negative somewhere.  Instead, we use the upper bound on gradient variables  to re-establish the density lower bound which do not require the initial conditions $\alpha(r,0)\geq0$ and $\beta(r,0)\geq0$.

We first take $\lambda=\frac{2}{\gamma-1}$ in \eqref{dd2} and \eqref{dd3} and denote
\begin{align}\label{dd11}
\hat{\alpha} = h^{-\frac{2}{\gamma - 1}}\alpha, \quad
\hat{\beta} = h^{-\frac{2}{\gamma - 1}}\beta,
\end{align}
to acquire
\begin{align}
\partial_1\hat{\beta} = &\frac{1+\gamma}{4}h^{\frac{2}{\gamma-1}}\hat{\beta}(\hat{\alpha}-\hat{\beta}) +A_1 \hat{\alpha} -B_1 \hat{\beta}
       +\frac{m u^2}{rc_2}\hat{\beta}, \label{dd4} \\
\partial_2\hat{\alpha} =&\frac{1+\gamma}{4}h^{\frac{2}{\gamma-1}}\hat{\alpha} (\hat{\beta}-\hat{\alpha}) +A_2\hat{\beta}-B_2 \hat{\alpha}+\frac{m u^2}{rc_1}\hat{\alpha}. \label{dd5}
\end{align}
Recalling Theorem \ref{thm_bounds} and \eqref{dd1}, one has for $r\geq b>0$
\begin{align}\label{dd81}
\frac{m u^2}{rc_2}+1\leq C_b,\quad \frac{m u^2}{rc_1}+1\leq C_b
\end{align}
for some positive constant $C_b$ depending on $b$. 

Then \eqref{dd4} and \eqref{dd5} can be rewritten as
\begin{align}
\partial_1\bar{\beta} = &\bigg(\frac{1+\gamma}{4}h^{\frac{2}{\gamma-1}}e^{C_bt}\bar{\beta}+A_1\bigg)(\bar{\alpha}-\bar{\beta}) \nonumber \\
&+(A_1 -B_1) \bar{\beta}+\bigg(\frac{m u^2}{rc_2}-C_b\bigg)\bar{\beta}, \label{dd6} \\
\partial_2\bar{\alpha} =&\bigg(\frac{1+\gamma}{4}h^{\frac{2}{\gamma-1}}e^{C_bt}\bar{\alpha}+A_2 \bigg) (\bar{\beta}-\bar{\alpha}) \nonumber \\
&+(A_2-B_2) \bar{\alpha}+\bigg(\frac{m u^2}{rc_1}-C_b\bigg)\bar{\alpha}. \label{dd7}
\end{align}
where
\begin{align}\label{dd8}
\bar{\alpha}=e^{-C_bt}\hat{\alpha},\quad \bar{\beta}=e^{-C_bt}\hat{\beta}.
\end{align}

Denote
\begin{align*}
M_0=
\left\{
\begin{array}{l}
\dspt \max_{r\in[b,\infty)}(\alpha, \beta)(r,0) \quad {\rm for\ Problem\ 1}, \\
\dspt \max\{\max_{r\in[b,\infty)}(\alpha, \beta)(r,0), \max_{t\in[0,T_0]}\alpha(B_b(t),t)\} \quad {\rm for\ Problem\ 2},
\end{array}
\right.
\end{align*}
and
$$
\bar{M}=1+(K\gamma)^{-\frac{1}{\gamma-1}}\frac{M_0}{\bar{\rho}},
$$
where $\bar{\rho}$ is defined as in Lemma \ref{lem3}.
Then we have
\begin{lemma}\label{eqbar}
Assume $1<\gamma<3$. Consider smooth solution on $t\in[0,T_0]$ for Problem 1 or 2 on
$\Omega$, satisfying the Assumption \ref{asu_1} on $(b,\infty)$, with $b>0$.  For Problem 2, we also assume Assumption \ref{asu_2} on the left boundary 1-characteristic $B_b(t)$.
Then there holds
\begin{align}\label{dd9}
\max_{\Omega\cap\{t\leq T_0\}}(\bar\alpha, \bar\beta)(r,t)< \bar{M}.
\end{align}
\end{lemma}
\begin{proof}
The proof is also based on the contradiction argument. In view of the construction of $\bar{M}$ and \eqref{dd11}, \eqref{dd8}, we first achieve that the initial or boundary data fulfill 
$$
\max_{[b,\infty)}(\bar\alpha,\bar\beta)(r,0)< \bar{M},\quad\hbox{and}\quad
\max_{t\geq0}\bar\alpha(B_b(t),t)< \bar{M}\quad \hbox{for Problem 2.}
$$
Assume that 
there exists some time, such that $\bar\alpha(r^*,t^*)=\bar{M}$ or $\bar\beta(r^*,t^*)=\bar{M}$, at some point $(r^*,t^*)$ with $0<t^*\leq T_0$. Then there exists
a characteristic triangle tip $\Pi_0$ or a characteristic quadrangle tip $\Pi_0$ at $(r_*,t_*)$ such that $\bar\alpha=\bar{M}$ or $\bar\beta=\bar{M}$ at  $(r_*,t_*)$, but  $\bar\alpha\leq \bar{M}$ and $\bar\beta\leq \bar{M}$ in the characteristic triangle or characteristic quadrangle.
Without loss of generality, we assume $\bar\alpha(r_*,t_*)=\bar{M}$ and $\bar\alpha< \bar{M}$ and $\bar\beta\leq \bar{M}$ on $\Pi_0\cap\{t<t_*\}$. Thus one first has
$$
\partial_2\bar{\alpha}(r_*,t_*)\geq0.
$$
On the other hand, we apply \eqref{BA12}, \eqref{dd81} and \eqref{dd7} to obtain
\begin{align}\label{dd10}
\partial_2\bar{\alpha}(r_*,t_*) \leq -\bar{\alpha}(r_*,t_*)<0,
\end{align}
which yields a contradiction. The proof of the lemma is finished.
\end{proof}

Now thanks to \eqref{dd11}, \eqref{dd9} and \eqref{inbo1}, we find that for $t\leq T$
\begin{align}\label{dd12}
\alpha(r,t), \beta(r,t) \leq e^{C_bT}[(\gamma-1)C_0]^{\frac{2}{\gamma-1}}\bar{M},
\end{align}
and then by \eqref{dd17}
\begin{align}\label{dd16}
\tilde\alpha(r,t), \tilde\beta(r,t) \leq e^{C_bT}[(\gamma-1)C_0]^{\frac{\gamma+1}{2(\gamma-1)}}\bar{M}.
\end{align}
Putting \eqref{dd12} into \eqref{aa7} and employing \eqref{aa8} arrives at
\begin{align}\label{dd13}
\partial_0\ln\bigg(\frac{1}{r^m\rho}\bigg)=&\frac{\alpha+\beta}{2} +\frac{mh^2}{2rc_1} +\frac{mh^2}{2rc_2} \nonumber \\
\leq & e^{C_bT_0}[(\gamma-1)C_0]^{\frac{2}{\gamma-1}}\bar{M} +\frac{2m(\gamma-1)}{b(3-\gamma)}C_0 =:\bar{M}_b.
\end{align}
We integrate \eqref{dd13} along $r_0(t)$ and utilize \eqref{aa9} to deduce
\begin{align}\label{dd14}
\rho(r,t)\geq \bigg(\frac{r_0(t_0)}{r}\bigg)^m\rho_0 e^{-\bar{M}_bt} \geq  \bar{\rho}\bigg(\frac{b}{b+2C_0t}\bigg)^m e^{-\bar{M}_bt}.
\end{align}

Based on the density lower bound in \eqref{dd14}, we now derive the singularity formation result.
We first rewrite \eqref{tildeab} and \eqref{tildeab2} as
\begin{align}
\partial_1\tilde{\beta} = &-\frac{1+\gamma}{8}h^{\frac{3-\gamma}{2(\gamma-1)}}\tilde{\beta}^2 \nonumber \\ &+\left\{-\frac{1+\gamma}{8}h^{\frac{3-\gamma}{2(\gamma-1)}}\tilde{\beta}^2 +\frac{3-\gamma}{4}\frac{m}{r}\frac{u^2}{c_2}\tilde{\beta} -B_1\tilde{\beta}+A_1\tilde{\alpha}\right\}\label{tildeab_3},  \\
\partial_2\tilde{\alpha} = &-\frac{1+\gamma}{8}h^{\frac{3-\gamma}{2(\gamma-1)}}\tilde{\alpha}^2 \nonumber \\  &+\left\{-\frac{1+\gamma}{8}h^{\frac{3-\gamma}{2(\gamma-1)}}\tilde{\alpha}^2 +\frac{3-\gamma}{4}\frac{m}{r}\frac{u^2}{c_1}\tilde{\alpha} -B_2\tilde{\alpha}+A_2\tilde{\beta}\right\}.\label{tildeab4}
\end{align}
We want to prove that when the initial $\tilde\beta$ or $\tilde\alpha$ is less than some threshold $N(b)$, then we always have
\beq\label{beta_ineq}
\partial_1\tilde{\beta}\leq -\frac{1+\gamma}{8}h^{\frac{3-\gamma}{2(\gamma-1)}}\tilde{\beta}^2,
\eeq
or
\beq\label{alpha_ineq}
\partial_2\tilde{\alpha} \leq -\frac{1+\gamma}{8}h^{\frac{3-\gamma}{2(\gamma-1)}}\tilde{\alpha}^2,
\eeq
by which, we can estimate the blowup time. In fact, integrating
\eqref{beta_ineq} along the 1-characteristic and assuming $\tilde\beta(r_1(t),t)<0$ for any time $t$ before blowup, we have
\begin{align}\label{aa48}
-\frac{1}{\tilde\beta(r, t)}\leq-\frac{1}{\tilde\beta(r^*,0)}
-\int_0^t \frac{1+\gamma}{8}h^{\frac{3-\gamma}{2(\gamma-1)}}(r_1(s),s)ds,
\end{align}
where $r=r_1(s)$ is the corresponding 1-characteristic curve and $r^*=r_1(0)$. Thanks to the lower bound of $\rho$ in \eqref{aa4}, one obtains
\begin{align}\label{aa49}
h=&\sqrt{K\gamma}\rho^{\frac{\gamma-1}{2}} \nonumber \\
\geq &\sqrt{K\gamma}\bar{\rho}^{\frac{\gamma-1}{2}}\bigg(\frac{b}{b+2C_0 t}\bigg)^{\frac{m(\gamma-1)}{2}} \exp\bigg(-\frac{\bar{M}_b(\gamma-1)}{2}t\bigg),
\end{align}
subsequently
\begin{align}\label{aa50}
h^{\frac{3-\gamma}{2(\gamma-1)}}(r,t)
\geq \widehat{C}\bigg(\frac{b}{b+2C_0 t}\bigg)^{\frac{m(3-\gamma)}{4}},
\end{align}
where
$$
\widehat{C}=(K\gamma)^{\frac{3-\gamma}{4(\gamma-1)}}\bar{\rho}^{\frac{3-\gamma}{4}} \exp\bigg(-\frac{\bar{M}_b(3-\gamma)}{4}T\bigg),
$$
with $t\leq T\leq T_0$. Combining \eqref{aa48} and \eqref{aa50} and using the fact $\frac{m(3-\gamma)}{4}<1$, we know that blowup happens not later than
\begin{align}\label{aa51}
t^*=\frac{b}{2C_0}\bigg\{\bigg(1+\frac{4C_0[4-m(3-\gamma)]} {-\tilde\beta(r^*,0)(\gamma+1)\widehat{C}b}\bigg)^{\frac{4}{4-m(3-\gamma)}}-1\bigg\}<T,
\end{align}
provided that $-\tilde\beta(r^*,0)$ is large enough. It is easy to see that $t^*\rightarrow 0$ as $-\tilde\beta(r^*,0)\rightarrow\infty$.

Noting $r\geq b>0$, the boundedness of $A_i$ and $B_i$ in \eqref{dd15}, the upper bound of $(\tilde\alpha, \tilde\beta)$ in \eqref{dd16}, and the density lower bound in \eqref{dd14},
we can choose $N(b)$ large enough such that
\[
\left(-\frac{1+\gamma}{8}h^{\frac{3-\gamma}{2(\gamma-1)}} \tilde{\beta}^2+\frac{3-\gamma}{4}\frac{m}{r}\frac{u^2}{c_2} \tilde{\beta}-B_1\tilde{\beta}+A_1\tilde{\alpha}\right)<0,
\]
when $\tilde{\beta}\leq -N(b)$ and $0\leq t< T$, and
\[
\left(-\frac{1+\gamma}{8}h^{\frac{3-\gamma}{2(\gamma-1)}} \tilde{\alpha}^2+\frac{3-\gamma}{4}\frac{m}{r}\frac{u^2}{c_1} \tilde{\alpha}-B_2\tilde{\alpha}+A_2\tilde{\beta}\right)<0,
\]
when $\tilde{\alpha}\leq -N(b)$ and $0\leq t< T$.
On the other hand, choose $N(b)$ large enough, such that, if $\min_{r\in[b,\infty)}(\tilde\alpha (r,0),\tilde\beta (r,0))<-N(b)$, then the blowup time is less than $T$. In addition, if $\min_{t\geq0}\tilde\alpha (B_b(t),t)<-N(b)$ for sufficiently large $N(b)$, we can similarly show the blowup time is finite.

In a summary, we proved the following singularity formation theorem:

\begin{theorem}\label{thm_sing}
We consider global solution with $C^1$ initial data satisfying the Assumption \ref{asu_1} on $(b,\infty)$, with $b>0$.  Assume $1<\gamma<3$ and there exists $M_0>0$, such that
$$
\max_{[b,\infty)}(\alpha,\beta)(r,0)< M_0,\quad\hbox{and}\quad
\max_{t\geq0}\alpha(B_b(t),t)< M_0\quad \hbox{for Problem 2,}
$$
For Problem 2, we also suppose that the $C^1$ boundary date satisfy the Assumption \ref{asu_2} on $B_b(t)$.
Then there exists a constant $N(b, T)$, depending on $b$ and $T$, such that,  if
$$
\min_{r\in[b,\infty)}\{\alpha(r,0),\beta(r,0\}\leq -N(b,T), \quad\hbox{or}\quad \min_{t\geq0}\alpha(B_b(t),t)\leq -N(b,T),
$$
then singularity forms before time $T$.
\end{theorem}

\section*{Acknowledgements}

The first and second author were partially supported by National Science Foundation (DMS-2008504, DMS-2306258), the third author was partially supported by National Natural
Science Foundation of China (12171130, 12071106), the fourth author was partially supported by National Science Foundation (DMS-2206218).


\begin{thebibliography}{21}

\bibitem{Ali} S. Alinhac, \emph{Blowup for Nonlinear Hyperbolic Equations}, Progress in Nonlinear
    Differential Equations and Their Applications 17, Birkh\"auser, Boston, 1995.

\bibitem{Bressan} A. Bressan, \emph{Hyperbolic systems of conservation laws: the
one-dimmensional Cauchy problem}, Oxford Lecture Ser. Math. Appl., Oxford Univ. Press, Oxford, 2000.

\bibitem{Vicol0} T. Buckmaster, T. Drivas, S. Shkoller and V. Vicol, Simultaneous development of
    shocks and cusps for 2D Euler with azimuthal symmetry from smooth data, \emph{Ann. PDE}, {\bf 8} (2022), no. 2, pp. 199.

\bibitem{Vicol1} T. Buckmaster, S. Shkoller and V. Vicol, Formation of shocks for 2d isentropic
    compressible Euler, \emph{Commun. Pure Appl. Math.}, {\bf 75} (2022), no. 9, 2069--2120.

\bibitem{Vicol2} T. Buckmaster, S. Shkoller and V. Vicol, Shock formation and vorticity creation for
    3d Euler, \emph{Commun. Pure Appl. Math.}, {\bf 76} (2023), no. 9, 1965--2072.

\bibitem{Vicol3} T. Buckmaster, S. Shkoller and V. Vicol, Formation of point shocks for 3D
    compressible Euler, \emph{Commun. Pure Appl. Math.}, {\bf 76} (2023), no. 9, 2073--2191.

\bibitem{CCW} H. Cai, G. Chen and T. Wang, Singularity formation for radially symmetric
    expanding wave of Compressible Euler Equations, \emph{SIAM J. Math. Anal.}, {\bf 55} (2023), no.
    4, 2917--2947.

\bibitem{G9} G. Chen, Optimal time-dependent lower bound on density for
classical solutions of 1-D compressible Euler equations, \emph{Indiana Univ. Math. J.}, {\bf 66}
(2017), no. 3, 725--740.


\bibitem{G10} G. Chen, Optimal time-dependent density lower bound for nonisentropic gas dynamics,
    \emph{J. Differential Equations}, {\bf 268} (2020), no. 7, 4017--4028.

\bibitem{CCZ} G. Chen, G.Q. Chen and S. Zhu, Formation of singularities and existence of
    global continuous solution for the compressible Euler equations, \emph{SIAM J. Math. Anal.},
{\bf 53} (2021), no. 6, 6280--6325.

\bibitem{CY} G. Chen and R. Young, Shock-free solutions of the compressible Euler equations, \emph{Arch. Ration. Mech. Anal.}, {\bf 217} (2015), no. 3, 1265--1293.

\bibitem{CYZ} G. Chen, R. Young and Q. Zhang, Shock formation in the compressible Euler equations
and related systems, \emph{J. Hyperbolic Differ. Equ.}, {\bf 10} (2013), no. 1, 149--172.

\bibitem{CPZ} G. Chen, R. Pan and S. Zhu, Singularity formation for compressible Euler
    equations,  \emph{SIAM J. Math. Anal.}, {\bf 49} (2017), no. 4, 2591--2614.

\bibitem{GQC} G.Q. Chen, Remarks on spherically symmetric solutions of the compressible Euler equations,
\emph{Proc Roy. Soc. Edinb. A}, {\bf 127} (1997), no. 2, 243--259.

\bibitem{Schen} S. Chen, Z. Xin and H. Yin, \emph{Formation and Construction of Shock Wave for
Quasilinear Hyperbolic System and its Application to Inviscid Compressible Flow}, Research Report, IMS,
CUHK, 1999.


\bibitem{Ch1} D. Christodoulou, \emph{The Shock Development Problem}, EMS Monographs in Mathematics,
European Mathematical Society (EMS), Z\"urich, 2019.

\bibitem{Ch2} D. Christodoulou and S. Miao, \emph{Compressible Flow and Euler's Equations}, Surveys of Modern Mathematics 9, International Press, Somerville, MA; Higher Education Press, Beijing, 2014.


\bibitem{courant} R. Courant and K. Friedrichs, \emph{Supersonic flow
and shock waves}, Wiley-Interscience, New York, 1948.




\bibitem{Di} M. Disconzi, C. Luo, G. Mazzone and J. Speck, Rough sound waves in 3D compressible Euler
    flow with vorticity, \emph{Sel. Math. New Ser.}, {\bf 28} (2022), article number 41.

\bibitem{Grassin} M. Grassin, Global smooth solutions to Euler equations for a perfect gas, \emph{Indiana
    Univ. Math. J.},  {\bf 47} (1998), no. 4, 1397--1432.

\bibitem{Godin} P. Godin, The lifespan of a class of smooth spherically symmetric solutions of the
    compressible Euler equations with variable entropy in three space dimensions, \emph{Arch. Ration. Mech.
    Anal.}, {\bf 177} (2005), no. 3, 479--511.


\bibitem{HJ} M. Had${\rm \breve{z}}$i${\rm \acute{c}}$ and J. Jang, Expanding large global solutions of the equations of compressible fluid mechanics, \emph{Inventiones Mathematicae}, {\bf 214} (2018), no. 3, 1205--1266.





\bibitem{John} F. John, Formation of singularities in one-dimensional nonlinear wave propagation,  \emph{Comm. Pure Appl. Math.}, {\bf 27} (1974), no. 3, 377--405.

\bibitem{Kong} D. Kong, Formation and propagation of singularities for 2x2 quasilinear hyperbolic systems, \emph{Trans. Amer. Math. Soc.}, {\bf 354} (2002), no. 4, 3155--3179.

\bibitem{Lai} G. Lai and M. Zhu, Global-in-time classical solutions to two-dimensional axisymmetric
    Euler system with swirl, \emph{Commun. Math. Sci.}, {\bf 21} (2023), No. 3, 829--857.


\bibitem{lax2} P. Lax, Development of singularities of solutions of nonlinear hyperbolic partial differential equations, \emph{J. Math. Phys.}, {\bf 5} (1964), no. 5, 611--614.


\bibitem{LiBook} T.T. Li, \emph{Global classical solutions for quasilinear hyperbolic systems}, Wiley, New york, 1994.

\bibitem{Li-Yu} T.T. Li and W.C. Yu, \emph{Boundary Value Problems for Quasilinear Hyperbolic
    Systems}, Duke University, 1985.

\bibitem{Li daqian} T.T. Li, Y. Zhou and D.X. Kong, Global classical solutions for general quasilinear hyperbolic systems with decay initial data, \emph{Nonlinear Analysis, Theory, Methods $\&$ Applications}, {\bf 28} (1997), no. 8, 1299--1332.

\bibitem{Liu1} T.P. Liu, The development of singularities in the nonlinear waves for quasi-linear hyperbolic partial differential equations, \emph{J. Differential Equations}, {\bf 33} (1979), no. 1, 92--111.

\bibitem{Luk1} J. Luk and J. Speck, Shock formation in solutions to the 2D compressible Euler
    equations in the presence of non-zero vorticity, \emph{Invent. Math.}, {\bf 214} (2018), no. 1, 1--169.

\bibitem{Luk2} J. Luk and J. Speck, The stability of simple plane-symmetric shock formation for 3D
    compressible Euler flow with vorticity and entropy, arXiv: 2107.03426, 2021.

\bibitem{LM} L. Magali, Global smooth solutions of Euler equations for Van der Waals gases,
    \emph{SIAM J. Math. Anal.}, {\bf 43} (2011), no. 2, 877--903.

\bibitem{Vicol5} I. Neal, C. Rickard, S. Shkoller and V. Vicol, A new type of stable shock formation
    in gas dynamics, arXiv: 2303.16842, 2023.

\bibitem{Vicol4} I. Neal, S. Shkoller and V. Vicol, A characteristics approach to shock formation in
    2D Euler with azimuthal symmetry and entropy, arXiv: 2302.01289, 2023.



\bibitem{PHJ} S. Parmeshwar, M. Had${\rm \breve{z}}$i${\rm \acute{c}}$ and J. Jang, Global expanding solutions of compressible Euler equations with small initial densities,  \emph{Quart. Appl. Math.}, {\bf 79} (2021), no. 2, 273--334.


\bibitem{Rickard1} C. Rickard, Global solutions to the compressible Euler equations with heat transport by convection around Dyson's isothermal affine solutions, \emph{Arch. Rat. Mech. Anal.}, {\bf 241} (2021), no. 2, 947--1007.


\bibitem{Rickard2} C. Rickard, M. Had${\rm \breve{z}}$i${\rm \acute{c}}$ and J. Jang, Global existence of the nonisentropic compressible Euler equations with vacuum boundary surrounding a variable entropy state, \emph{Nonlinearity}, {\bf 34} (2020), no. 1, 33.


\bibitem{Shkoller-Sideris} S. Shkoller and T. Sideris, Global existence of near-affine solutions to the compressible Euler equations, \emph{Arch. Rat. Mech. Anal.}, {\bf 234} (2019), no. 1, 115--180.

\bibitem{sideris} T. Sideris, Formation of singularities in three-dimensional compressible fluids, \emph{Commun. Math. Phys.}, {\bf 101} (1985), no. 4, 475--485.

\bibitem{sideris1} T. Sideris, Global existence and asymptotic behavior of affine motion of 3D ideal
    fluids surrounded by vacuum, \emph{Arch. Ration. Mech. Anal.}, {\bf 225} (2017), no. 1, 141--176.

\bibitem{Zhu} C. Zhu, Global smooth solution of the nonisentropic gas dynamics system, \emph{Proc. Roy. Soc. Edinb. Sect. A}, {\bf 126} (1996), no. 4, 768--775.



%
%
%
%



\end{thebibliography}
\end{document}